\documentclass[11pt,reqno]{amsart}
\usepackage[colorlinks=true,
pdfstartview=FitV, linkcolor=cyan, citecolor=magenta,
urlcolor=blue]{hyperref}
\usepackage{amsmath,amsfonts,latexsym,amssymb}
\usepackage{mathrsfs}
\usepackage[latin1]{inputenc}
\usepackage[T1]{fontenc}
\usepackage{ae,aecompl}
\usepackage{braket}
\usepackage{comment}
\usepackage{color}
\usepackage{graphicx}
\usepackage{subfig}
\usepackage{bbm}

\usepackage{enumitem}

\newtheorem{theorem}{Theorem}[section]
\newtheorem{lemma}[theorem]{Lemma}
\newtheorem{proposition}[theorem]{Proposition}
\newtheorem{corollary}[theorem]{Corollary}

\newtheorem*{thm*}{\protect\theoremname}
\theoremstyle{definition}
\newtheorem{definition}[theorem]{Definition}
\newtheorem{remark}[theorem]{Remark}
\newtheorem{example}[theorem]{Example}
\newtheorem*{cor*}{\protect\corollaryname}

\long\def\@savemarbox#1#2{\global\setbox#1\vtop{\hsize\marginparwidth 
  \@parboxrestore\tiny\raggedright #2}}
\renewcommand{\d}{{\rm d}}

\newcommand{\Cc}{{\mathcal C}}

\newcommand{\GG}{\mathsf G}

\newcommand{\PSL}{\mathsf{PSL}}

\newcommand{\mfg}{{\mathfrak g}}
\newcommand{\mfa}{{\mathfrak a}}
\newcommand{\mfk}{{\mathfrak k}}
\newcommand{\mfp}{{\mathfrak p}}
\newcommand{\mfu}{{\mathfrak u}}

\newcommand{\abs}[1]{\left|#1\right|}

\newcommand{\Oc}{\mathcal O}

\newcommand{\Pc}{\mathcal P}
\newcommand{\Zc}{\mathcal Z}

\DeclareMathOperator{\id}{\mathrm{id}}

\DeclareMathOperator{\Ab}{\mathbb{A}}

\DeclareMathOperator{\Cb}{\mathbb{C}}

\DeclareMathOperator{\Hb}{\mathbb{H}}

\DeclareMathOperator{\Nb}{\mathbb{N}}
\DeclareMathOperator{\Pb}{\mathbb{P}}
\DeclareMathOperator{\Rb}{\mathbb{R}}

\DeclareMathOperator{\Zb}{\mathbb{Z}}

\DeclareMathOperator{\Usf}{\mathsf{U}}
\DeclareMathOperator{\Psf}{\mathsf{P}}
\DeclareMathOperator{\Ksf}{\mathsf{K}}

\DeclareMathOperator{\Fc}{\mathcal{F}}

\DeclareMathOperator{\SL}{\mathsf{SL}}
\newcommand{\norm}[1]{\left\|#1\right\|}
\newcommand{\ip}[1]{\left\langle #1\right\rangle}

\usepackage{ mathrsfs }

\providecommand{\corollaryname}{Corollary}
\providecommand{\theoremname}{Theorem}

\begin{document}

\title{Patterson--Sullivan measures for relatively Anosov groups}
\author[Canary]{Richard Canary}
\address{University of Michigan}
\author[Zhang]{Tengren Zhang}
\address{National University of Singapore}
\author[Zimmer]{Andrew Zimmer}
\address{University of Wisconsin-Madison}
\thanks{Canary was partially supported by grant  DMS-1906441 from the National Science Foundation.
Zhang was partially supported by the NUS-MOE grant A-8000458-00-00. Zimmer was partially supported by a Sloan research fellowship and grants DMS-2105580 and DMS-2104381 from the
National Science Foundation.}

\begin{abstract} 
We establish existence, uniqueness and ergodicity results for Patterson--Sullivan measures for relatively Anosov groups. 
As applications we obtain an entropy gap theorem and a strict concavity result for entropies associated to linear functionals. 
\end{abstract} 

\maketitle

\tableofcontents

\section{Introduction}

Anosov subgroups and relatively Anosov subgroups of semisimple Lie groups are respectively natural generalizations of convex cocompact subgroups and geometrically finite subgroups of rank one semisimple Lie groups to the higher rank setting. Patterson--Sullivan measures for Anosov subgroups have been extensively studied, see \cite{dey-kapovich, sambarino-dichotomy, BLLO, lee-oh-invariant}. 
They have also been studied for relatively Anosov subgroups of the projective general linear group which preserve a properly convex domain, see~\cite{blayac, BZ, Bray, zhu-ergodicity,BrayTiozzo}.  More generally, Patterson--Sullivan measures can be constructed for transverse subgroups, of which Anosov subgroups, relatively Anosov subgroups, and discrete subgroups of rank one semisimple Lie groups are examples, see \cite{CZZ3,KOW1, KOW2}. For a more detailed historical discussion of Patterson--Sullivan measures for discrete subgroups of semisimple Lie groups see~\cite[Sec.\ 1.1]{CZZ3}. 

In this paper we study Patterson--Sullivan measures for relatively Anosov subgroups of semisimple Lie groups. The key new result here is that the Poincar\'e series associated to such a Patterson--Sullivan measure always diverges at its critical exponent if this critical exponent is finite. As a consequence, we establish uniqueness and ergodicity results for such Patterson--Sullivan measures. We then derive an entropy gap theorem and a strict concavity result for the entropy.  

\bigskip

Delaying precise definitions until Sections~\ref{sec: ss Lie group background},~\ref{sec:rel hyp groups background} and~\ref{sec: discrete subgroup background}, we informally introduce the notation
necessary to state our main results. 
For the entire paper, $\GG$ will be a connected semisimple Lie group without compact factors and with finite center. We fix a Cartan decomposition $\mathfrak{g} = \mathfrak{k} + \mathfrak{p}$ 
of the Lie algebra of $\GG$, a Cartan subspace $\mfa \subset \mfp$, and a Weyl chamber $\mathfrak{a}^+\subset\mathfrak a$. Let $\Delta \subset \mfa^*$ be the corresponding system of simple restricted roots, and let $\kappa:\GG\to\mfa^+$ denote the associated Cartan projection.
Given a subset 
$\theta \subset \Delta$, we let $\Psf_\theta \subset \GG$ denote the associated parabolic subgroup and let $\Fc_\theta=\GG / \Psf_\theta$ denote the associated flag manifold. We will always assume that $\theta$ is symmetric.

Suppose $\Gamma \subset \GG$ is a non-elementary discrete subgroup which, as an abstract group, is relatively hyperbolic with respect to 
a finite collection $\Pc$   of subgroups in $\Gamma$. Let $\partial(\Gamma,\Pc)$ denote its associated Bowditch boundary.
Then  $\Gamma$ is \emph{$\Psf_\theta$-Anosov relative to $\Pc$} if $\alpha\circ\kappa$ is proper on $\Gamma$ for all $\alpha\in\theta$, its limit set $\Lambda_\theta(\Gamma)$ in 
$\Fc_\theta$ consists of mutually transverse flags and there exists a continuous $\Gamma$-equivariant map
$$
\xi : \partial(\Gamma, \Pc) \rightarrow \Fc_\theta
$$
which is a homeomorphism onto $\Lambda_\theta(\Gamma)$. 

The action of $\GG$ on $\Fc_\theta$ preserves a vector valued cocycle $B_\theta : \GG \times \Fc_\theta \rightarrow \mfa_\theta$, called the \emph{partial Iwasawa cocycle}, whose image lies in a subspace $\mfa_\theta \subset \mfa$ associated to $\theta$. If $\phi \in \mfa_\theta^*$, then the composition $\phi \circ B_\theta$ is a real valued cocycle, which allows us to define Patterson--Sullivan measures associated to such $\phi$. More precisely, if $\phi \in \mfa_\theta^*$ and $\Gamma \subset \GG$ is $\Psf_\theta$-Anosov relative to $\Pc$, then a \emph{$\phi$-Patterson--Sullivan measure for $\Gamma$ of dimension $\beta$} is a Borel probability measure $\mu$ on $\Fc_\theta$ where
\begin{enumerate}
\item ${\rm supp}(\mu) \subset \Lambda_\theta(\Gamma)$, 
\item for every $\gamma \in \Gamma$ the measures $\gamma_* \mu$, $\mu$ are absolutely continuous and 
$$
\frac{d\gamma_* \mu}{d \mu}(F) = e^{-\beta \phi (B_\theta(\gamma^{-1}, F))}
$$
for $\mu$-almost every $F \in \Fc_\theta$. 
\end{enumerate} 

In the case when $\GG = \mathsf{SO}_0(d,1)$ is the group of orientation-preserving isometries of real hyperbolic $d$-space $\Hb^d$, there is a single simple restricted root $\Delta =\{\alpha\}$ and 
$\Fc_{\alpha}$ naturally identifies with the geodesic boundary of $\Hb^d_{\Rb}$. Further (up to scaling), $\alpha \circ B_\alpha$ identifies with the usual Busemann cocycle. 
Since a discrete subgroup $\Gamma$ of $\mathsf{SO}_0(d,1)$ is relatively $\Psf_\alpha$-Anosov if and only if it is geometrically finite,  the above definition encompasses the classical notion of Patterson--Sullivan measures for geometrically finite Kleinian groups. 

As in the classical theory, there is an associated Poincar\'e series and critical exponent. Let $\kappa_\theta : \GG \rightarrow \mfa_\theta$ denote the \emph{partial Cartan projection} 
defined in Section~\ref{sec: ss Lie group background}. Then given $\phi \in \mfa_\theta^*$, $s>0$ and a discrete group $\Gamma \subset \GG$, the \emph{$\phi$-Poincar\'e series} is 
$$
Q_\Gamma^\phi(s) := \sum_{\gamma \in \Gamma} e^{-s\phi(\kappa_\theta(\gamma))}.
$$
The \emph{$\phi$-critical exponent}, denoted $\delta^\phi(\Gamma) \in [0,+\infty]$, is the critical exponent of the above series, that is $Q_\Gamma^\phi(s)$ converges 
when $s > \delta^\phi(\Gamma) $ and diverges when $0<s < \delta^\phi(\Gamma)$. 
Given $\phi \in \mfa_\theta^*$ let $\bar{\phi} \in  \mfa_\theta^*$ be the unique functional where $\bar{\phi}(\kappa_\theta(g)) = \phi(\kappa_\theta(g^{-1}))$ for all $g\in \GG$. 
Notice that $Q^\phi_\Gamma = Q^{\bar{\phi}}_\Gamma$ and so $\delta^\phi(\Gamma)=\delta^{\bar{\phi}}(\Gamma)$. 

The main result of this paper is that  if $\Gamma$ is relatively Anosov and its $\phi$-critical exponent is finite, then its $\phi$-Poincar\'e series diverges at its critical exponent.

\begin{theorem}[Theorem \ref{rel anosov divergence}]
\label{intro: rel anosov divergent}
If $\Gamma \subset \GG$ is a $\Psf_\theta$-Anosov subgroup relative to $\Pc$, $\phi \in \mfa_\theta^*$ and $\delta^\phi(\Gamma) < +\infty$,
then $Q_\Gamma^\phi$ diverges at its critical exponent.
\end{theorem}

\begin{remark} As mentioned above, in this paper we assume that relatively Anosov groups are non-elementary. Theorem~\ref{rel anosov divergence} holds for elementary transverse groups (see Remark~\ref{rmk:main theorem in elementary case}), however the non-elementary assumption is necessary for many of the applications of the theorem. \end{remark} 

Later, we will discuss some important consequences of Theorem \ref{intro: rel anosov divergent}, and also provide an outline of its proof (in Sections \ref{consequences} and \ref{outline} respectively). Notice that Theorem \ref{intro: rel anosov divergent} fails in the setting of transverse groups, since the Poincar\'e series of any finitely generated, 
geometrically infinite discrete subgroup of $\mathsf{SO}(3,1)$ whose limit
set is not all of $\partial\mathbb H^3$ converges at its critical exponent (see \cite[Cor.\ 4.2]{canary-laplacian} and \cite[Thm.\ 2.17,\, Cor.\ 2.18]{sullivan-positivity}).

We also provide a characterization of the functionals with finite critical exponent, surprisingly the only requirement is that $\phi(\kappa_\theta(\cdot))$ converges to infinity along escaping sequences in the group. We also show that if $\Psf_\theta$ contains no simple factors of $\GG$, then $\phi(\kappa_\theta(\cdot))$ has linear lower and upper bounds in terms of the distance 
$\d_M$ on the Riemannian symmetric space $M$ associated to $\GG$.

\begin{theorem}[see Section~\ref{sec:characterizing finite entropy functionals}]\label{thm:characterizing finite entropy functionals} Suppose $\Gamma\subset \GG$ is a $\Psf_\theta$-Anosov subgroup relative to $\Pc$ and  $\phi\in \mathfrak{a}_\theta^*$. The following are equivalent: 
\begin{enumerate}
\item $\lim_{n \rightarrow \infty} \phi(\kappa_\theta(\gamma_n))=+\infty$ for every sequence of distinct elements $\{\gamma_n\}$ in $\Gamma$.
\item $\delta^\phi(\Gamma) < +\infty$. 
\end{enumerate}
Moreover, if $\Psf_\theta$ contains no simple factors of $\GG$, then the above conditions are equivalent to: 
\begin{enumerate}
\setcounter{enumi}{2}
\item If $x_0\in M$ , there exist constants  $c\ge 1$ and $C > 0$ such that 
$$
  \frac{1}{c}\d_M(\gamma(x_0),x_0)- C \le\phi(\kappa_\theta(\gamma)) \le c\d_M(\gamma(x_0),x_0)+C
$$
for all $\gamma \in \Gamma$. 
\end{enumerate} 
\end{theorem}

We will observe later, see Section \ref{sec: example}, that in Theorem \ref{thm:characterizing finite entropy functionals} , Condition (3) is not equivalent to Conditions (1) and (2) without the assumption that $\Psf_\theta$ contains no simple factors of $\GG$. Also, in the spirit of Sambarino's analogous result for Anosov groups \cite{sambarino-dichotomy}, we observe that Theorem~\ref{thm:characterizing finite entropy functionals} implies that
$\phi\in\mathfrak{a}_\theta^*$ has finite critical exponent if and only if  $\phi$ is  positive on the $\theta$-Benoist limit cone, see Section~\ref{sec:characterizing finite entropy functionals}.

\subsection{Consequences of Theorem \ref{intro: rel anosov divergent}}\label{consequences}  We now recall some results from our previous work~\cite{CZZ3} from which we can derive several consequences of Theorem \ref{intro: rel anosov divergent}.

A discrete subgroup  $\Gamma\subset\GG$ is \emph{$\Psf_\theta$-transverse} if $\alpha\circ\kappa$ is proper on $\Gamma$ for all $\alpha\in\theta$ and  its limit set $\Lambda_\theta(\Gamma)$ in 
$\Psf_\theta$ consists of mutually transverse flags. A $\Psf_\theta$-transverse subgroup acts on this limit set as a convergence group, so one can define the set of conical limit points $\Lambda^{\rm con}_\theta(\Gamma) \subset \Lambda_\theta(\Gamma)$ using the standard convergence group action definition. In this setting, we established in \cite{CZZ3} the following analogue of the Hopf-Tsuji-Sullivan dichotomy.

\begin{theorem}[{\cite[Thm.\ 1.4]{CZZ3}}]\label{thm:dichotomy}  Suppose $\Gamma\subset\GG$ is a non-elementary $\Psf_\theta$-transverse subgroup, $\phi\in \mathfrak{a}^*_\theta$ and $\delta:=\delta^\phi(\Gamma) < +\infty$. 
\begin{itemize}
\item If 
$Q_\Gamma^\phi(\delta)=+\infty$, 
then there exists a unique $\phi$-Patterson--Sullivan measure $\mu_\phi$  for $\Gamma$ of dimension $\delta$ and there exists a unique $\bar{\phi}$-Patterson--Sullivan measure $\mu_{\bar{\phi}}$  for $\Gamma$ of dimension $\delta$. Moreover:
\begin{enumerate}
\item $\mu_\phi( \Lambda_\theta^{\rm con}(\Gamma)) = 1=\mu_{\bar{\phi}}( \Lambda_\theta^{\rm con}(\Gamma))$. 
\item The action of $\Gamma$ on $(\Lambda_\theta(\Gamma), \mu_\phi)$ and  $(\Lambda_\theta(\Gamma), \mu_{\bar{\phi}})$  is ergodic. 
\item The action of $\Gamma$ on $(\Lambda_\theta(\Gamma)^2, \mu_{\bar{\phi}} \otimes \mu_{\phi})$ is ergodic. 
\end{enumerate} 
\item If $Q_\Gamma^\phi(\delta) < +\infty$, then $\mu( \Lambda_\theta^{\rm con}(\Gamma)) = 0$ for any $\phi$-Patterson--Sullivan measure $\mu$ for $\Gamma$. 
\end{itemize}
\end{theorem}

If $\Gamma$ is a $\Psf_\theta$-Anosov subgroup relative to $\Pc$, then $\Gamma$ is $\Psf_\theta$-transverse. Thus, we obtain, as a consequence of Theorem \ref{intro: rel anosov divergent} and Theorem \ref{thm:dichotomy}, ergodicity and uniqueness results for their Patterson--Sullivan measures. These generalize results earlier obtained for Anosov groups (see \cite{dey-kapovich, sambarino-dichotomy, BLLO, lee-oh-invariant}).

\begin{corollary}
\label{rel anosov hst}
Suppose $\Gamma \subset \GG$ is a $\Psf_\theta$-Anosov subgroup relative to $\Pc$, $\phi \in \mfa_\theta^*$ and $\delta:=\delta^\phi(\Gamma) < +\infty$. Then:
\begin{enumerate}
\item There is a unique $\phi$-Patterson--Sullivan measure $\mu_\phi$ for $\Gamma$ of dimension $\delta$ and a unique $\bar{\phi}$-Patterson--Sullivan measure $\mu_{\bar{\phi}}$ for $\Gamma$ of dimension $\delta$.
\item $\Gamma$ acts ergodically on $(\Lambda_\theta(\Gamma), \mu_\phi)$ and $(\Lambda_\theta(\Gamma), \mu_{\bar{\phi}})$.
\item $\Gamma$ acts ergodically on $(\Lambda_\theta(\Gamma) \times \Lambda_\theta(\Gamma), \mu_\phi \otimes \mu_{\bar{\phi}})$. 
\end{enumerate}
\end{corollary} 

In \cite{CZZ3}, we also established a criterion for when the critical exponent of a subgroup  of a transverse group is strictly less than that of the entire group. 

\begin{theorem}[{\cite[Thm.\ 4.1]{CZZ3}}]\label{thm: entropy gap CZZ3}
Suppose $\Gamma\subset\GG$ is a non-elementary $\Psf_\theta$-transverse subgroup, $\phi\in\mathfrak{a}_\theta^*$ and $\delta^\phi(\Gamma) < +\infty$. 
If $\Gamma_0$ is a subgroup of $\Gamma$ such that $Q_{\Gamma_0}^\phi$ diverges at its critical exponent and $\Lambda_\theta(\Gamma_0)$ is a proper subset of $\Lambda_\theta(\Gamma)$, then $\delta^\phi(\Gamma)>\delta^\phi(\Gamma_0).$
\end{theorem}

Suppose that $\Gamma\subset\GG$ is $\Psf_\theta$-Anosov relative to $\Pc$. A subgroup $\Gamma_0\subset\Gamma$ is \emph{relatively quasiconvex} if its action on its limit set $\Lambda(\Gamma_0)\subset \partial(\Gamma,\Pc)$ is geometrically finite. We prove that non-elementary, relatively quasiconvex subgroups of $\Gamma$ are themselves $\Psf_\theta$-Anosov relative to $\Pc$ (see Proposition~\ref{prop:rel quasiconvex are rel Anosov}), and that the limit set of any infinite index quasiconvex subgroup of $\Gamma$ is a closed, proper subset of the limit set of $\Gamma$ (see Lemma \ref{lem:limit set is a proper subset}). Combining these facts with Theorems \ref{intro: rel anosov divergent} and \ref{thm: entropy gap CZZ3}, we may then prove the following result for relatively quasiconvex subgroups of infinite index.

\begin{corollary}[Corollary \ref{cor:entropy gap rel anosov 1}]
\label{cor:entropy gap rel anosov}
Suppose $\Gamma \subset \GG$ is a $\Psf_\theta$-Anosov subgroup relative to $\Pc$, and $\phi \in \mfa_\theta^*$ such that $\delta^\phi(\Gamma) < +\infty$. 
If $\Gamma_0$ is an infinite index relatively quasiconvex subgroup of $(\Gamma, \Pc)$, then
$$\delta^\phi(\Gamma)>\delta^\phi(\Gamma_0).$$
\end{corollary}

Finally, in \cite{CZZ3}, we established that the critical exponent is strictly concave on the space of linear functionals which diverge at their critical exponent, except when
there is agreement of length functionals. More precisely, given $\phi \in \mfa_\theta^*$ and $g \in \GG$, the \emph{$\phi$-length} of $g$ is 
$$
\ell^\phi(g) : = \lim_{n \rightarrow \infty} \frac{1}{n} \phi(\kappa_\theta(g^n)).
$$

\begin{theorem}[{\cite[Thm.\ 1.5]{CZZ3}}]\label{thm:manhattan curve}
Suppose $\Gamma$ is a non-elementary $\mathsf{P}_{\theta}$-transverse subgroup of $\GG$ and $\phi_1,\phi_2 \in \mathfrak{a}_\theta^*$ satisfy $\delta^{\phi_1}(\Gamma) = \delta^{\phi_2}(\Gamma)=1$. If $\phi = \lambda \phi_1 + (1-\lambda) \phi_2$ for some $\lambda \in (0,1)$, then  
$$
\delta:=\delta^\phi(\Gamma) \le 1.
$$
Moreover, if $Q_{\Gamma}^\phi(\delta)=+\infty$, then equality occurs if and only if $\ell^{\phi_1}(\gamma) = \ell^{\phi_2}(\gamma)$ for all $\gamma \in \Gamma$.  
\end{theorem}

Together, Theorem~\ref{intro: rel anosov divergent} and Theorem \ref{thm:manhattan curve} give the following result.

\begin{corollary} 
\label{rel anosov concavity}
Suppose $\Gamma \subset \GG$ is a $\Psf_\theta$-Anosov subgroup relative to $\Pc$ and $\phi_1,\phi_2 \in \mfa_\theta^*$ satisfy $\delta^{\phi_1}(\Gamma) = \delta^{\phi_2}(\Gamma)=1$. If $\phi = \lambda \phi_1 + (1-\lambda)\phi_2$ for some $\lambda \in (0,1)$, then 
$$
\delta^\phi(\Gamma) \le 1
$$
with equality if and only if $\ell^{\phi_1}(\gamma)=\ell^{\phi_2}(\gamma)$ for all $\gamma \in \Gamma$. 
\end{corollary}

In the Zariski dense case, see \cite[Cor.\ 1.6]{CZZ3}, a theorem of Benoist \cite{benoist-asymptotic} implies that distinct length functions cannot agree.

\begin{corollary}\label{cor: manahattan curve Z dense case} Suppose $\Gamma \subset \GG$ is a Zariski dense $\Psf_\theta$-Anosov subgroup relative to $\Pc$, and $\phi_1,\phi_2 \in \mfa_\theta^*$ are distinct and satisfy $\delta^{\phi_1}(\Gamma) = \delta^{\phi_2}(\Gamma)=1$. If $\phi = \lambda \phi_1 + (1-\lambda)\phi_2$ for some $\lambda \in (0,1)$, then $\delta^\phi(\Gamma) < 1$.
\end{corollary}

\subsection{Outline of the proof of Theorem~\ref{intro: rel anosov divergent}} \label{outline}
The strategy of the proof of Theorem~\ref{intro: rel anosov divergent} is inspired by earlier work of Blayac--Zhu~\cite{BZ} in the context of relatively hyperbolic groups preserving properly convex domains and Patterson--Sullivan measures defined using the Busemann functions associated to the Hilbert distance. The key technical result needed to prove Theorem~\ref{intro: rel anosov divergent} is that if $\Gamma$ is $\Psf_\theta$-Anosov relative to $\Pc$ and $\delta^\phi(\Gamma)<+\infty$, then the $\phi$-Poincar\'e series of any peripheral
subgroup diverges at its critical exponent.

\begin{theorem}[Theorem \ref{thm:entropy gap for rel Anosov}]
\label{peripheral divergent}
 Suppose $\Gamma \subset \GG$ is a $\Psf_\theta$-Anosov subgroup relative to $\Pc$, $\phi \in \mfa_\theta^*$ and $\delta^\phi(\Gamma) < +\infty$.
If $P\in\Pc$, then $Q_P^\phi$ diverges at its critical exponent.
\end{theorem}

Together, Theorems~\ref{thm: entropy gap CZZ3} and~\ref{peripheral divergent} imply that $\delta^\phi(P) < \delta^\phi(\Gamma)$ for all $P\in\Pc$. We may then adapt arguments of Dal'bo--Otal--Piegn\'e \cite{DOP} to our setting to conclude Theorem \ref{intro: rel anosov divergent}.

The proof of Theorem \ref{peripheral divergent} makes use of Hironaka's famous result on the resolution of singularities. More precisely, in Section~\ref{sec: relating to integral}, we relate the Poincar\'e series $Q_P^\phi(s)$ associated to a peripheral subgroup $P$ to an integral of the form 
\begin{equation}\label{eqn:dumb integral}
 \int_{\Rb^n} \big(R_1^{ \ell_1} \cdots R_m^{\ell_m}\big)^{-s} d\lambda 
 \end{equation}
where $\lambda$ is the Lebesgue measure, $R_1,\dots, R_m : \Rb^n \rightarrow \Rb$ are positive rational functions that are defined everywhere and $\ell_1,\dots, \ell_m$ are real numbers such that $R_1^{ \ell_1} \cdots R_m^{\ell_m}$ is a proper function. In Section~\ref{sec: resolution of singularities}, we use Hironaka's resolution of singularities to show that the integral in Equation~\eqref{eqn:dumb integral} diverges at its critical exponent. 
This in turn implies that the Poincar\'e series associated to any peripheral subgroup diverges at its critical exponent. 

 In Blayac and Zhu's setting, the integral in Equation~\eqref{eqn:dumb integral} involves a single rational function (i.e. \ $m=1$) and one can deduce that it diverges at its critical exponent from a result  of  Benoist--Oh~\cite[Prop.\ 7.2]{benoist-oh} (see the proof of ~\cite[Lem.\ 8.9]{BZ}). 
The case when $m > 1$ is more technical and our use of  Hironaka's resolution of singularities to understand the integral at infinity is motivated by Benoist and Oh's arguments.

\subsection{Conditions (1) and (2) are not equivalent to (3) in Theorem \ref{thm:characterizing finite entropy functionals} }\label{sec: example}

We now give an example which shows that Condition (3) in Theorem \ref{thm:characterizing finite entropy functionals}
may fail to be equivalent to  Conditions (1) and (2) when $\Psf_\theta$ contains a simple factor of $\GG$.

\begin{example}\label{ex:stupid example} Let $\GG = \PSL(2,\Rb) \times \PSL(2,\Rb)$. Note that
$$
\mathfrak{a}: = \left\{ {\rm diag}(a_1,\dots,a_4) : a_1+a_2=0=a_3+a_4\right\}
$$
is a Cartan subspace of $\GG$ and $\Delta : =\{\alpha_1, \alpha_1^\prime\}\subset\mathfrak a^*$ is a system of simple restricted roots, where 
$$
\alpha_1({\rm diag}(a_1,\dots,a_4)) = a_1 - a_2 \quad \text{and} \quad \alpha_1^\prime({\rm diag}(a_1,\dots,a_4)) = a_3 - a_4.
$$
Set $\theta: =\{ \alpha_1\}$, and note that
$$
\mfa_\theta = \{ {\rm diag}(a, -a, 0, 0) : a \in \Rb\}
$$
and 
$$
\kappa_\theta( (g,h)) = (\log \sigma_1(g), \log \sigma_2(g), 0, 0)
$$
where $\sigma_1(g) \ge \sigma_2(g)$ are the singular values of some (any) lift of $g$ to $\SL(2,\Rb)$. Also,
$$
\Psf_\theta = \Psf_1 \times \PSL(2,\Rb)
$$
where $\Psf_1 \subset \PSL(2,\Rb)$ stabilizes the line $\Rb \cdot \begin{pmatrix} 1 \\ 0 \end{pmatrix}$. In particular, $\Psf_\theta$ contains a simple factor of $\GG$. 

Let $S$ be a thrice punctured sphere equipped with a complete hyperbolic structure such that one of its ends is a cusp while the other two ends are hyperbolic funnels. Then there is a discrete subgroup $\Gamma_0\subset\PSL(2,\Rb)$ such that $S = \Gamma_0 \backslash \Hb^2$. Let $h$ be a hyperbolic element in $\PSL(2,\Rb)$ whose conjugacy class $[h]$ corresponds to an oriented geodesic going once around one of the hyperbolic funnels of $S$, and let $u$ be a unipotent element in $\PSL(2,\Rb)$ whose conjugacy class $[u]$ corresponds to an oriented closed curve going once around the cusp of $S$. We may assume that $\Gamma_0=\ip{u}*\ip{h}$. Then $\Gamma_0$ is $\Psf_1$-Anosov relative to $\Pc_0:=\{ \ip{u}\}$. Thus, if we set $\Gamma : = \ip{ \gamma, \eta} \subset \GG$, where 
\[\gamma := (u,h) \in \GG\quad\text{ and }\quad\eta := (h, \id) \in \GG,\]
then $\Gamma$ is $\Psf_\theta$-Anosov relative to $\Pc : = \{ \ip{\gamma}\}$. Furthermore,
$$
\lim_{n \rightarrow \infty} \alpha_1(\kappa_\theta(g_n))=+\infty
$$ 
for every sequence of distinct elements $\{g_n\}$ in $\Gamma$, so (1) in Theorem~\ref{thm:characterizing finite entropy functionals} holds. 

On the other hand, the Riemannian symmetric space $M$ associated to $\GG$ is the product 
\[M = \Hb^2_{\Rb} \times \Hb^2_{\Rb}\] 
where $\Hb^2_{\Rb}$ is the real hyperbolic 2-space. Then for any $p_0=(x_0, y_0) \in M$ we have 
$$
\liminf_{n \rightarrow \infty} \frac{1}{n} \log \d_M(\gamma^n(p_0), p_0) =  \liminf_{n \rightarrow \infty} \frac{1}{n} \log \d_{\Hb_{\Rb}^2}(h^n(y_0), y_0) > 0
$$
since $h$ is hyperbolic and $u$ is unipotent. However,  
$$
\limsup_{n \rightarrow \infty} \frac{1}{n} \log \phi(\kappa_\theta(\gamma^n)) = \limsup_{n \rightarrow \infty} \frac{1}{n} \log \frac{\sigma_1(u^n)}{\sigma_2(u^n)}=0
$$
since $u$ is unipotent. So (3) in Theorem~\ref{thm:characterizing finite entropy functionals} does not hold. 
\end{example} 

\subsection*{Acknowledgements} We thank the referee for their comments and their careful reading of the original manuscript.

\section{Semisimple Lie groups}\label{sec: ss Lie group background}

In this expository section we introduce some of the notation used throughout the paper. All the notation introduced is the same as in \cite{CZZ3}.

\textbf{As in the introduction (and the rest of the paper)}, let $\GG$ be a connected semisimple Lie group without compact factors and with finite center. Fix a Cartan decomposition
\[\mfg = \mfk \oplus \mfp\] 
of the Lie algebra $\mfg$ of $G$, a Cartan subspace $\mathfrak a\subset\mathfrak p$, and a Weyl chamber by $\mfa^+\subset\mfa$. Let $\Sigma^+$ be the set of positive restricted roots associated to $\mathfrak a^+$, let $\Delta\subset\mathfrak a^*$ be the system of simple restricted roots associated to $\mathfrak a^+$, and let $\Ksf\subset\GG$ denote the maximal compact Lie subgroup whose Lie algebra is $\mfk$. 

\subsection{Cartan projection}
The \emph{Cartan projection} is the map $\kappa : \GG \rightarrow \mfa^+$ with the defining property that $\kappa(g) \in \mfa^+$ is the unique element such that 
$$
g = m e^{\kappa(g)} \ell
$$
for some $m, \ell \in \Ksf$ (in general $m$ and $\ell$ are not uniquely determined by $g$). Such a decomposition $g = m e^{\kappa(g)} \ell$ is called a \emph{$\mathsf{KAK}$-decomposition} of $g$, see \cite[Chap.\ IX, Thm.\ 1.1]{Helgason}. 

There is a unique involutive linear map $\iota : \mfa \rightarrow \mfa$, called the \emph{opposition involution}, such that $\iota(-\mfa^+) = \mfa^+$ and 
$$
\iota( \kappa(g)) = \kappa(g^{-1})
$$
for all $g \in \GG$. The dual of $\iota$ satisfies 
$$
\iota^*(\Delta) = \Delta.
$$

As in the introduction we let $M $ denote the standard Riemannian symmetric space associated to $\GG$, and let $\d_M$ denote the distance function on $M$ induced by the Riemannian metric. As a $\GG$-space, $M=\GG / \Ksf$, and $\d_M$ satisfies 
$$
\d_M(g\Ksf, h\Ksf) = \norm{\kappa(g^{-1}h)} \quad \text{for all} \quad g,h \in \GG,
$$
where $\norm{\cdot}$ is the norm on $\mathfrak a$ induced by the restriction of the Killing from to $\mathfrak a$. The Cartan projection satisfies the following estimates. 

\begin{lemma}[{see e.g.\ \cite[Lem.\ 2.3]{Kassel2008}}]\label{lem:kappa multiplication estimate} If $g,h \in \GG$, then 
$$
\norm{\kappa(gh)-\kappa(h)} \le \norm{\kappa(g)} \quad \text{and} \quad \norm{\kappa(gh)-\kappa(g)} \le \norm{\kappa(h)}.
$$
\end{lemma}

\subsection{Parabolic subgroups and flag manifolds} Given a subset $\theta \subset \Delta$, the \emph{parabolic subgroup associated to $\theta$}, denoted by $\Psf_\theta \subset \GG$, is the normalizer of 
 $$
\mfu_\theta:=  \bigoplus_{\alpha \in \Sigma^+_\theta} \mfg_\alpha
$$
where $\Sigma^+_\theta := \Sigma^+ - {\rm Span}( \Delta - \theta)$. Then the \emph{flag manifold associated to $\theta$} is
\[\Fc_\theta:=\mathsf{G}/\mathsf{P}_\theta.\]

 In this paper will only consider the case when $\theta \subset \Delta$ is \emph{symmetric}, that is $\iota^*(\theta) = \theta$. In this case, there is a unique open $\GG$-orbit in 
$$
\Fc_\theta \times \Fc_\theta
$$
and we say that two flags $F_1, F_2 \in \Fc_\theta$ are \emph{transverse} if $(F_1, F_2)$ is in this orbit. For a flag $F\in\Fc_\theta$, we let 
$$
\Zc_F \subset \Fc_\theta
$$
denote the set of flags that are not transverse to $F$.

Following the notation in~\cite{GGKW}, we define a map 
\[U_\theta: \GG \rightarrow \Fc_\theta\] 
by fixing a $\mathsf{KAK}$-decomposition $g = m_g e^{\kappa(g)} \ell_g$ for each $g \in \GG$ and then letting $U_\theta(g) := m_g \Psf_\theta$. One can show that if $\alpha(\kappa(g)) > 0$ for all $\alpha \in \theta$, then $U_\theta(g)$ is independent of the choice of $\mathsf{KAK}$-decomposition, see \cite[Chap.\ IX, Thm.\ 1.1]{Helgason}, and hence $U_\theta$ is continuous on the set 
$$
\left\{ g \in \GG : \alpha(\kappa(g)) > 0 \text{ for all } \alpha \in \theta\right\}.
$$ 

The action of $\GG$ on $\Fc_\theta$ has the following north-south like dynamics. 

\begin{proposition}[{see e.g.\ \cite[Prop.\ 2.3]{CZZ3}}]\label{prop:characterizing convergence in general symmetric case} Suppose $\theta \subset \Delta$ is symmetric, $F^\pm \in \Fc_\theta$ and $\{g_n\}$ is a sequence in $\GG$. The following are equivalent:  
\begin{enumerate}
\item $U_\theta(g_n) \rightarrow F^+$, $U_\theta(g_n^{-1}) \rightarrow F^-$ and $\lim_{n \rightarrow \infty} \alpha(\kappa(g_n)) = \infty$ for every $\alpha \in \theta$,
\item $g_n(F) \to F^+$ for all $F \in \Fc_\theta - \Zc_{F^-}$, and this convergence is uniform on compact subsets of $\Fc_\theta - \Zc_{F^-}$. 
\item $g_n^{-1}(F) = F^-$ for all $F \in \Fc_\theta - \Zc_{F^+}$, and this convergence is uniform on compact subsets of $\Fc_\theta - \Zc_{F^+}$. 
\item there are open sets $\mathcal{U}^\pm\subset\Fc_\theta$ such that $g_n(F) \to F^+$ for all $F \in \mathcal U^+$ and $g_n^{-1}(F) \to F^-$ for all $F \in \mathcal U^-$.
\end{enumerate}
\end{proposition}

\subsection{Fundamental weights and partial Cartan projections}

For any $\alpha \in \Delta$, let $\omega_\alpha \in \mfa^*$ denote the associated fundamental weight. 

Given a subset $\theta \subset \Delta$, the \emph{partial Cartan subspace associated to $\theta$} is 
$$
\mfa_\theta := \{ H \in \mfa : \alpha(H) = 0 \text{ for all } \alpha \in \Delta - \theta\}.
$$
Then there is a unique projection 
\[p_\theta : \mfa \rightarrow \mfa_\theta\] 
such that $\omega_\alpha(X)=\omega_\alpha(p_\theta(X))$ for all $\alpha\in\theta$ and $X\in\mathfrak a$. The \emph{partial Cartan projection associated to $\theta$} is
\[\kappa_\theta:=p_\theta\circ\kappa:\GG\to\mathfrak a_\theta.\]
One can show that $\{ \omega_\alpha|_{\mfa_\theta} : \alpha \in \theta\}$ is a basis of $\mfa_\theta^*$ and hence we will identify 
$$
\mfa_\theta^*={\rm Span}\{ \omega_\alpha : \alpha \in \theta\}\subset\mathfrak a^*.
$$
Note that $\omega_\alpha(\kappa_\theta(g))=\omega_\alpha(\kappa(g))$ for all $\alpha\in\theta$ and $g\in\GG$ and so 
\begin{equation}
\label{eqn:stupid equality between projections} 
\phi(\kappa_\theta(g))=\phi(\kappa(g))
\end{equation} 
for all $\phi \in \mfa_\theta^*$ and $g\in\GG$.

\subsection{The partial Iwasawa cocycle} Let $\Usf := \exp( \mfu_\Delta)$. The \emph{Iwasawa decomposition} states that the map 
$$
(k,a, u) \in \Ksf \times \exp(\mfa) \times \Usf \mapsto kau \in \GG
$$
is a diffeomorphism, see \cite[Chap.\ VI, Prop.\ 6.46]{Knapp}. Using this, Quint \cite{quint-ps} defined the \emph{Iwasawa cocycle}
\[B : \GG \times \Fc_\Delta \rightarrow \mfa\] 
with the defining property that $gk \in \Ksf \cdot \exp(B(g,F)) \cdot \Usf$ for all $(g,F)\in \GG \times \Fc_\Delta$, where $k\in\Ksf$ is an element such that $F = k \Psf_\Delta$. 

For any $\theta\subset\Delta$, note that $\Psf_\Delta\subset \Psf_\theta$, so the identity map on $\GG$ induces a surjection \hbox{$\Pi_\theta:\Fc_\Delta\to\Fc_\theta$}. 
The \emph{partial Iwasawa cocycle} is the map 
\[B_\theta: \GG \times \Fc_\theta \rightarrow \mfa_\theta\] 
defined by $B_\theta(g,F) = p_\theta( B(g,F') )$ for some (all) $F'\in\Pi_\theta^{-1}(F)$.
By~\cite[Lem.\ 6.1 and 6.2]{quint-ps}, this is a well-defined cocycle, that is 
$$
B_\theta(gh, F) = B_\theta(g, hF) + B_\theta(h, F)
$$
for all $g,h \in \GG$ and $F \in \Fc_\theta$. 

\subsection{The Linear Case}\label{sec:linear case} We now briefly describe the above notations when $\GG = \PSL(d,\Rb)$. Let $\mathfrak{sl}(d,\Rb) = \mathfrak{k}  + \mathfrak{p} $ denote the standard Cartan decomposition of $\mathfrak{sl}(d,\Rb)$, that is 
\begin{align*}
\mathfrak{k} = \{ X \in \mathfrak{sl}(d,\Rb) : {^tX} = -X\}\quad\text{and}\quad \mathfrak{p} = \{ X \in \mathfrak{sl}_d(\Rb) : {^tX} = X\}.
\end{align*}
Also let 
$$
\mathfrak{a}:=\{{\rm diag}(a_1,\ldots,a_d)\in\mathfrak{sl}(d,\Rb) :\ a_1+\cdots+a_d=0 \}\subset  \mathfrak{p}
$$
denote the standard Cartan subspace with the standard positive Weyl chamber 
$$
\mathfrak{a}^+:=\{{\rm diag}(a_1,\ldots,a_d)\in\mathfrak{a} :\ a_1 \ge  \dots \ge a_d \}\subset  \mathfrak{p}.
$$
Then the opposite involution is given by 
$$
\iota({\rm diag}(a_1,\ldots,a_d)) ={\rm diag}(-a_d,\ldots,-a_1)
$$
and the Cartan projection is given by 
$$
\kappa(g)=(\log\sigma_1(g),\cdots,\log\sigma_d(g))
$$ 
where $\sigma_1(g) \ge \cdots \ge \sigma_d(g)$ are the singular values of any lift of $g$ to $\mathsf{SL}(d,\Rb)$.

The standard system of simple restricted roots is $\Delta :=\{ \alpha_1,\dots, \alpha_{d-1}\} \subset \mathfrak{a}^*$ where
$$
\alpha_j( {\rm diag}(a_1,\ldots,a_d)) = a_j-a_{j+1}
$$
for all ${\rm diag}(a_1,\ldots,a_d) \in \mathfrak{a}$. Then the fundamental weights are given by 
$$
\omega_{\alpha_j}( {\rm diag}(a_1,\ldots,a_d))  = a_1+\dots + a_j
$$
and so $\iota^*(\omega_{\alpha_j}) = \omega_{\alpha_{d-j}}$. 

Let $e_1,\dots, e_d$ denote the standard basis of $\Rb^d$ and let $F_0^j := \ip{e_1,\dots, e_j}$ be the subspace spanned by $e_1,\dots, e_j$. Then given $\theta=\{ \alpha_{j_1}, \dots, \alpha_{j_n}\} \subset \Delta$ with $j_1 < j_2 < \dots < j_n$, the parabolic subgroup $\Psf_\theta \subset \PSL(d,\Rb)$ is the stabilizer of the flag 
$$
( F_0^{j_1}, F_0^{j_2}, \cdots, F^{j_n}_0)
$$
and $\Fc_\theta$ is the partial flag manifold 
$$
\mathcal F_\theta  = \left\{ (F^{j_1},\dots, F^{j_n}) :\ \mathrm{dim}\left(F^{j_i}\right)=j_i  \text{ and }F^{j_1}\subset F^{j_2} \subset\cdots\subset F^{j_n}\right\}.
$$ 
In the case when $\theta$ is symmetric, two flags $F_1, F_2 \in \mathcal F_\theta$ are transverse if and only if they are transverse in the usual sense: $F_1^j \oplus F_2^{d-j} = \Rb^d$ for all $\alpha_j \in \theta$. 

In the linear case we often replace subscripts of the form $\theta = \{ \alpha_{j_1}, \dots, \alpha_{j_n}\}$ with $j_1,\cdots ,j_n$. So for instance 
$$
\Fc_{1,d-1} = \Fc_{\{\alpha_1,\alpha_{d-1}\}}
$$
denotes the line/hyperplane partial flag manifold. 

\subsection{Properties of unipotent subgroups}
Recall that a $d\times d$ real matrix $A$ is \emph{unipotent} if 1 is the only eigenvalue of $A$ (over $\Cb$) and a subgroup $\mathsf{U} \subset \mathsf{GL}(d,\Rb)$ is called \emph{unipotent} if every element in $\mathsf{U}$ is unipotent. When $\GG$ has trivial center, a subgroup $\mathsf{U} \subset \GG$ is \emph{unipotent} if ${\rm Ad}(\mathsf{U}) \subset \SL(\mathfrak{g})$ is unipotent. Given a Lie subgroup $\mathsf{H} \subset \mathsf{G}$, the \emph{unipotent radical} of $\mathsf{H}$ is the maximal connected normal unipotent subgroup of $\mathsf{H}$.

Unipotent subgroups have the following well-known properties. 

\begin{proposition}\label{prop:properties of unipotent subgroups}
Suppose $Z(\mathsf{G})$ is trivial and $\mathfrak{u}$ is the Lie algebra of a connected unipotent subgroup $\mathsf{U} \subset \GG$. Then: 
\begin{enumerate}
\item The exponential map induces a diffeomorphism $\mathfrak{u} \rightarrow \mathsf{U}$. 
\item Let $\lambda$ be a  measure on $\mathfrak{u}$ which is obtained by pulling back the Lebesgue measure on $\Rb^{\dim \mathfrak{u}}$ via a linear isomorphism (notice that any two measures obtained this way are scalar multiplies of each other). Then the push-forward $\mu : = \exp_*\lambda$ is a Haar measure on $\mathsf{U}$.
\item If $\Phi : \GG \rightarrow \SL(d,\Rb)$ is a representation, then $\Phi(\mathsf{U}) \subset \SL(d,\Rb)$ is unipotent. 
\item There exists $C > 0$ such that
$$
\norm{\kappa(\exp(Y))} \le C+C\log (1+\norm{Y})
$$
for all $Y \in \mathfrak{u}$. 
\item  For any $\alpha \in \Delta$ there exist $M_\alpha \in \Nb$, $C_\alpha>1$ and a positive everywhere defined rational function $R_\alpha: \mathfrak{u} \rightarrow \Rb$ where 
$$
\frac{1}{C_\alpha} R_\alpha(Y)^{1/M_\alpha} \le e^{\omega_\alpha(\kappa(\exp(Y)))} \le C_\alpha R_\alpha(Y)^{1/M_\alpha}
$$
for all $Y \in \mathfrak{u}$.
 
\end{enumerate} 
\end{proposition}  

We provide a proof of Proposition \ref{prop:properties of unipotent subgroups} in Appendix \ref{app: unipotent subgroups}.

\section{Relatively hyperbolic groups}\label{sec:rel hyp groups background} 

In this expository section we recall one of the many equivalent definitions of a relatively hyperbolic group, for more background and other definitions see~\cite{Bowditch_relhyp, DS2005, Osin, Yaman}.

\subsection{Relatively hyperbolic groups}\label{sec: defining rel hyp groups} Recall that the action, by homeomorphisms, of a  group $\Gamma$ on a 
compact metrizable space $M$ is said to be a (discrete)  {\em convergence group action} if whenever
$\{\gamma_n\}$ is a sequence of distinct elements in $\Gamma$, then there are points $x,y\in M$ and a subsequence $\{\gamma_{n_j}\}$ so that $\gamma_{n_j}(z)$ converges to  $x$ for all $z\in M-\{y\}$ and the convergence is uniform on compact subsets of $M-\{y\}$.

Suppose that $\Gamma$ acts on $M$ as a convergence group, then: 
\begin{itemize}
 \item A point $x \in M$ is a \emph{conical limit point} if there exist $a, b \in M$ distinct and a sequence $\{\gamma_n\}$ in $\Gamma$ such that $\gamma_n(x) \to a$ and $\gamma_n(y) \to b$ for all $y \in M - \{x\}$. 
 \item An element $\gamma \in \Gamma$ is \emph{parabolic} if  it has infinite order and fixes exactly one point in $M$.
 \item A infinite order subgroup $H \subset \Gamma$ is \emph{parabolic} if it fixes some point of $M$ and each infinite order element in $H$ is parabolic. The fixed point of a parabolic subgroup is called a \emph{parabolic point}.
 \item A parabolic point $x \in M$ is \emph{bounded} if the quotient ${\rm Stab}_\Gamma(x) \backslash (M - \{x\})$ is compact.

\end{itemize}
Finally, $\Gamma$ is called a \emph{geometrically finite convergence group} if every point in $M$ is either a conical limit point or a bounded parabolic point.

\begin{definition}\label{defn:RH}
Given a finitely generated group $\Gamma$ and a collection $\Pc$ of finitely generated infinite subgroups, we say that $\Gamma$ is \emph{hyperbolic relative to $\Pc$}, or that $(\Gamma,\Pc)$ is \emph{relatively hyperbolic}, if $\Gamma$ acts on a compact perfect metrizable space $M$ as a geometrically finite convergence group and the maximal parabolic subgroups are exactly the set
$$
 \{ \gamma P \gamma^{-1} : P \in \Pc, \gamma \in \Gamma\}.
$$
To avoid confusion with parabolic subgroups in the Lie group sense, we also sometimes call a maximal parabolic subgroup of $\Gamma$ a \emph{peripheral subgroup}.
\end{definition}

\begin{remark}
Notice that by definition we assume that a relatively hyperbolic group is non-elementary (i.e.\  $M$ is perfect and hence is infinite) and finitely generated. \end{remark}

By a theorem of Bowditch \cite[Thm.\ 9.4]{Bowditch_relhyp}, given a relatively hyperbolic group $(\Gamma,\Pc)$, any two compact perfect metrizable spaces satisfying Definition~\ref{defn:RH} are $\Gamma$-equivariantly homeomorphic. 
This unique topological space is then denoted by $\partial(\Gamma,\Pc)$ and called the \emph{Bowditch boundary of $(\Gamma, \Pc)$}.

\subsection{The Groves--Manning cusp space}  
Given a relatively hyperbolic group $(\Gamma,\mathcal P)$, 
Groves and Manning \cite{GrovesManning} constructed  a Gromov hyperbolic space which $\Gamma$ acts on proper discontinuously so that subgroups
in $\mathcal P$ preserve ``combinatorial horoballs.'' These spaces, now called Groves--Manning cusp spaces, are central tools in the study of relatively hyperbolic groups.
We recall their construction here.

\begin{definition} Suppose $Y$ is a graph with the simplicial distance $\d_Y$. The \emph{combinatorial horoball} $\mathcal{H}(Y)$ is the graph, also equipped with the simplicial distance, that has vertex set $Y^{(0)} \times \Nb$ and two types of edges:
	\begin{itemize}
		\item \emph{vertical edges} joining vertices $(v,n)$ and $(v,n+1)$, 
		\item \emph{horizontal edges} joining vertices $(v,n)$ and $(w,n)$ when $\d_Y(v,w) \le 2^{n-1}$. 
	\end{itemize}
\end{definition} 

\begin{definition} \label{def:cusp spaces} \label{defn: cusped Cayley graph}
Suppose that $(\Gamma,\Pc)$ is relatively hyperbolic. A finite symmetric generating set $S \subset \Gamma$ is \emph{adapted} if $S \cap P$ is a generating set of $P$ for every $P \in \Pc$. Given such an $S$, we let $\Cc(\Gamma, S)$ and $\Cc(P, S \cap P)$ denote the associated Cayley graphs. Then the associated \emph{Groves--Manning cusp space}, denoted $\Cc_{GM}(\Gamma, \Pc, S)$, is obtained from the Cayley graph $\Cc(\Gamma, S)$ by attaching, for each $P \in \Pc$ and each coset $\gamma P \in \Gamma / P $, a copy of the combinatorial horoball $\mathcal{H}( \gamma\Cc(P, S \cap P))$  by identifying $\gamma\Cc(P, S \cap P)\subset\Cc(\Gamma, S)$ with the $n=1$ level of $\mathcal{H}( \gamma\Cc(P, S \cap P))$.

\end{definition}

\begin{theorem}[{\rm{Groves--Manning} \cite[Thm.\ 3.25]{GrovesManning}}] If $(\Gamma, \Pc)$ is relatively hyperbolic and $S$ is an adapted finite generating set, then 
$\Cc_{GM}(\Gamma, \Pc, S)$ is a proper geodesic Gromov hyperbolic space such that
\begin{enumerate}
\item $\Gamma$ acts properly discontinuously on $\Cc_{GM}(\Gamma, \Pc, S)$ by isometries, 
\item
every point in $X$ is within a uniformly bounded distance of a bi-infinite geodesic, and
\item there exists a $\Gamma$-equivariant homeomorphism between $\partial_\infty \Cc_{GM}(\Gamma, \Pc, S)$, the Gromov boundary of $\Cc_{GM}(\Gamma, \Pc, S)$, 
and $\partial(\Gamma, \Pc)$.
\end{enumerate}
\end{theorem}

\section{Discrete subgroups of semisimple Lie groups}\label{sec: discrete subgroup background}

In this expository section we introduce three classes of discrete subgroups in $\GG$ and state some of their basic properties. \textbf{In the rest of the paper}, we  assume that $\theta \subset \Delta$ is symmetric.

\subsection{Divergent groups} A discrete subgroup $\Gamma \subset \GG$ is called \emph{$\Psf_\theta$-divergent} if 
$$
\lim_{n \rightarrow +\infty} \min_{\alpha \in \theta} \alpha(\kappa_\theta(\gamma_n)) = +\infty
$$
whenever $\{\gamma_n\}$ is a sequence of distinct elements in $\Gamma$. The \emph{limit set} $\Lambda_\theta(\Gamma)$ of such a subgroup is the set of accumulation points of $\{ U_\theta(\gamma) : \gamma \in \Gamma\}$. We note that in the literature, divergent groups are sometimes called regular groups (e.g.\ \cite{KLP1}).

The limit set of a divergent group can be used to compactify it.

\begin{lemma}[{see e.g.\ \cite[Prop.\ 2.3]{CZZ3}}]\label{lem: limit set compactifies} 
If $\Gamma \subset \GG$ is $\Psf_\theta$-divergent, then the set $\Gamma \cup \Lambda_\theta(\Gamma)$ has a topology that makes it a compactification of $\Gamma$. More precisely:
\begin{enumerate}
\item $\Gamma \cup \Lambda_\theta(\Gamma)$ is a compact metrizable space.
\item If $\Gamma$ has the discrete topology, then $\Gamma \hookrightarrow \Gamma \cup \Lambda_\theta(\Gamma)$ is an embedding.
\item If $\Lambda_\theta(\Gamma)$ has the subspace topology from $\Fc_\theta$, then $\Lambda_\theta(\Gamma) \hookrightarrow \Gamma \cup \Lambda_\theta(\Gamma)$ is an embedding.
\item A sequence $\{\gamma_n\}$ in $\Gamma$ converges to $F$ in $\Lambda_\theta(\Gamma)$ if and only if 
\[
\lim_{n \rightarrow +\infty} \min_{\alpha \in \theta} \alpha(\kappa_\theta(\gamma_n)) = +\infty\quad\text{and}\quad U_\theta(\gamma_n) \rightarrow F.
\] 
\item The natural left action of $\Gamma$ on $\Gamma \cup \Lambda_\theta(\Gamma)$ is by homeomorphisms. 
\end{enumerate}
\end{lemma}

\subsection{Transverse groups}  A $\Psf_\theta$-divergent subgroup $\Gamma\subset\GG$ is \emph{$\Psf_\theta$-transverse} if $\Lambda_\theta(\Gamma)$ is a transverse subset of $\Fc_\theta$, i.e. distinct pairs of flags in $\Lambda_\theta(\Gamma)$ are transverse. We note that in the literature, transverse groups are sometimes called regular antipodal groups (e.g.\ \cite{KLP1}). 

One crucial feature of $\Psf_\theta$-transverse groups is that  they act as a convergence group on their limit sets.

\begin{proposition}\cite[Prop.\ 5.38]{KLP2}\label{prop: convergence group}
If $\Gamma$ is $\Psf_\theta$-transverse, then $\Gamma$ acts on $\Lambda_\theta(\Gamma)$ as a convergence group. 
In particular, if $\Gamma$ is non-elementary, then $\Gamma$ acts on $\Lambda_\theta(\Gamma)$ minimally, and $\Lambda_\theta(\Gamma)$ is perfect. 
\end{proposition}

When $\Gamma \subset \GG$ is $\Psf_\theta$-transverse, the set of conical limit points for the action of $\Gamma$ on $\Lambda_\theta(\Gamma)$ is called the {\em $\theta$-conical limit set} and is denoted $\Lambda^{\rm con}_\theta(\Gamma)$.

\subsection{Relatively Anosov subgroups} There are several equivalent definitions of relatively Anosov groups. The definition we use comes from~\cite{KL}. 

A $\Psf_\theta$-transverse subgroup $\Gamma \subset \GG$ is  \emph{$\Psf_\theta$-Anosov relative to $\Pc$}, a finite collection of subgroups of $\Gamma$, if $(\Gamma, \Pc)$ is relatively hyperbolic with Bowditch boundary $\partial(\Gamma, \Pc)$ and there is a continuous $\Gamma$-equivariant map
$$
\xi \colon \partial(\Gamma,\Pc) \to \Fc_\theta
$$
which is a homeomorphism onto $\Lambda_{\theta}(\Gamma)$. Observe that such a $\xi$ is unique, so we refer to it as the \emph{limit map} of $\Gamma$.

The next result shows that this limit maps plays nicely with the Gromov boundary of a Groves--Manning cusp space. 

\begin{proposition}\label{prop:compactifications are the same} Suppose  $\Gamma \subset \GG$ is $\Psf_\theta$-Anosov relative to $\Pc$, with limit map $\xi : \partial (\Gamma,\Pc) \rightarrow  \Fc_\theta$. Let $X$ be a Groves--Manning cusp space for $(\Gamma, \Pc)$,  and let $b_0 \in X$. If $\{\gamma_n\}$ is a sequence in $\Gamma$ and $\gamma_n(b_0) \rightarrow x \in \partial_\infty X = \partial (\Gamma,\Pc)$, then $U_\theta(\gamma_n) \rightarrow \xi(x)$. 
\end{proposition} 

\begin{proof}Since $\Fc_\theta$ is compact, it suffices to show that every convergent subsequence of $\{U_\theta(\gamma_n)\}$ converges to $\xi(x)$. Suppose $U_\theta(\gamma_{n_j}) \rightarrow F^+$. Passing to a further subsequence we can suppose that $\gamma_{n_j}^{-1}(b_0) \rightarrow y \in \partial (\Gamma,\Pc)$ and $U_\theta(\gamma_{n_j}^{-1}) \rightarrow F^-$. Then, by properties of Gromov hyperbolic spaces, $\gamma_{n_j}(z) \rightarrow x$ for all $z \in \partial (\Gamma,\Pc) - \{ y\}$. Also, by Proposition~\ref{prop:characterizing convergence in general symmetric case},
$$
\lim_{j \rightarrow \infty} \gamma_{n_j}(F) = F^+
$$ 
for all $F \in \Fc_\theta$ transverse to $F^-$.

Notice that Proposition~\ref{prop:characterizing convergence in general symmetric case} implies that $F^- \in \Lambda_\theta(\Gamma)$ and so $F^- = \xi(y^\prime)$ for some $y^\prime \in \partial(\Gamma, \Pc)$. Fix $z \in \partial (\Gamma,\Pc) -\{y,y^\prime\}$. Then 
$$
\xi(x) = \lim_{j \rightarrow \infty} \xi(\gamma_{n_j}(z)) = \lim_{n \rightarrow \infty} \gamma_{n_j} \xi(z) = F^+
$$
since $\xi(z)$ is transverse to $F^-=\xi(y^\prime)$. 
\end{proof}

%
%

The following theorem was established in ~\cite{zhu-zimmer1} when $\GG = \SL(d,\Rb)$. In Appendix~\ref{appendix:properties of pers} we will explain why it is also true in the following setting. In the following theorem and elsewhere in the paper, given a Lie group $\mathsf{H}$ let $\mathsf{H}^0$ denote the connected component of the identity in $\mathsf{H}$. 

\begin{theorem}\label{thm:properties of relatively Anosov representations} Assume $Z(\mathsf{G})$ is trivial and $\Psf_\theta$ contains no simple factors of $\GG$. Suppose $\Gamma\subset \GG$ is a non-elementary $\Psf_\theta$-Anosov subgroup relative to $\Pc$.
\begin{enumerate}
\item If $X$ is a Groves--Manning cusp space for $(\Gamma, \Pc)$ and $M :=\GG/\Ksf$ is a Riemannian symmetric space associated to $\GG$, then there exist $c > 1$, $C> 0$ such that 
$$
\frac{1}{c} \d_M(\gamma\Ksf, \Ksf)- C \le \d_X(\gamma, \id) \le c \d_M(\gamma\Ksf, \Ksf) + C
$$
for all $\gamma \in \Gamma$. 
\item If $P \in \Pc$, then $P$ is a cocompact lattice in a closed Lie subgroup $\mathsf{H} \subset \mathsf{G}$ with finitely many components. Moreover,
\begin{enumerate} 
\item $\mathsf{H} = \mathsf{L} \ltimes \mathsf{U}$ where $\mathsf{L}$ is compact and $\mathsf{U}$ is the unipotent radical of $\mathsf{H}$. 
\item $\mathsf{H}^0 = \mathsf{L}^0 \times \mathsf{U}$ and $\mathsf{L}^0$ is Abelian. 

\end{enumerate} 
\end{enumerate} 

\end{theorem}

Example~\ref{ex:stupid example} provides an example where $\Psf_\theta$ contains a simple factor of $\GG$ and the conclusions of Theorem~\ref{thm:properties of relatively Anosov representations} fail.

\subsection{Helpful reductions}\label{sec: a helpful reduction}  We first explain why one can often reduce to the case where the center $Z(\GG)$ of $\GG$ is trivial and $\Psf_\theta$ contains no simple factors of $\GG$.  We then explain how one can often reduce to the case where $\GG=\SL(d,\Rb)$.

Decompose the Lie algebra $\mfg$ of $G$ into a product of simple Lie algebras, $\mfg = \oplus_{j=1}^m \mfg_j$. For each $1 \le j \le m$, let $\GG_j \subset \GG$ denote the closed connected normal subgroup with Lie algebra $\mfg_j$. Then 
$$
\GG = \GG_1 \cdots \GG_m
$$
is an almost direct product and $\GG_1, \dots, \GG_m$ are called the \emph{simple factors} of $\GG$. 

\begin{proposition}[{\cite[Prop.\ 2.9]{CZZ3}}]\label{prop:reducing to no simple factors} Suppose $\theta \subset \Delta$ is symmetric and $\mathsf{H} : = Z(\GG) \prod \{ \GG_j : \GG_j \subset \Psf_\theta\}$. Let $p : \GG \rightarrow \GG^\prime := \GG / \mathsf{H}$ be the quotient map. Then:
\begin{enumerate}
\item $\GG^\prime$ is a semisimple Lie group without compact factors and with trivial center. 
\item There is a Cartan decomposition $\mfg^\prime = \mfk^\prime + \mfp^\prime$ of the Lie algebra of $\GG^\prime$, a Cartan subspace $\mfa^\prime \subset \mfp^\prime$, a system of simple restricted roots $\Delta^\prime \subset (\mfa^\prime)^*$ and a subset $\theta^\prime \subset \Delta^\prime$ such that 
$$
p(\Psf_\theta) = \Psf^\prime_{\theta^\prime}
$$
(where $\Psf^\prime_{\theta^\prime}$ is the parabolic subgroup of $\GG^\prime$ associated to $\theta^\prime$). Moreover, $\Psf_{\theta^\prime}^\prime$ contains no simple factors of $\GG^\prime$. 
\item $\d p$ induces an isomorphism of the partial Cartan subspaces $\mfa_\theta$ and $\mfa_{\theta^\prime}^\prime$. Moreover, the partial Cartan projections satisfies
$$
\d p(\kappa_\theta(g)) = \kappa_{\theta^\prime}^\prime(p(g)) \quad \text{for all $g \in \GG$}. 
$$
\item The map $\xi : \Fc_\theta \rightarrow \Fc_{\theta^\prime} = \GG^\prime / \Psf_{\theta^\prime}^\prime$ defined by $\xi(g\Psf_\theta) = p(g) \Psf_{\theta^\prime}^\prime$ is a diffeomorphism which preserves transversality. Moreover, the partial Iwasawa cocycles satisfies
$$
\d p(B_\theta(g,F)) = B^\prime_{\theta^\prime}(p(g), \xi(F)) \quad \text{for all $g \in \GG$ and $F \in \Fc_\theta$}.
$$
\end{enumerate}

\end{proposition}

Using the discussion in~\cite[Section 3]{GGKW} it is possible to prove the following result which allows one to reduce many calculations to the linear case, see ~\cite[Prop.\ B.1]{CZZ3} for details. The statement of the result uses the notation introduced in Section~\ref{sec:linear case}.

\begin{proposition}\label{prop:reduction to the linear case} For any symmetric $\theta \subset \Delta$ and $\chi \in \sum_{\alpha \in \theta} \Nb \cdot \omega_\alpha$ there exist $d \in \Nb$,  an irreducible linear representation $\Phi : \GG \rightarrow \mathsf{SL}(d,\Rb)$ and a $\Phi$-equivariant smooth embedding
$$
\xi : \Fc_\theta \rightarrow \Fc_{1,d-1}(\Rb^d)
$$
such that: 
\begin{enumerate}
 \item $F_1, F_2 \in \Fc_\theta$ are transverse if and only if $\xi(F_1)$ and $\xi(F_2)$ are transverse. 
 \item There exists $N \in \Nb$ such that 
$$
\log\sigma_1(\Phi(g)) =N\chi(\kappa(g))
$$
for all $g \in \GG$. 
\item $\alpha_1(\kappa(\Phi(g))) = \min_{\alpha \in \theta} \alpha( \kappa(g))$ for all $g \in \GG$.
\item If $\min_{\alpha \in \theta} \alpha(\kappa(g)) > 0$, then 
$$
\xi( U_\theta(g)) = U_{1,d-1}(\Phi(g)). 
$$
\item  $\Gamma \subset \GG$ is $\Psf_{\theta}$-divergent (respectively $\Psf_{\theta}$-transverse) if and only if $\Phi(\Gamma)$ is $\Psf_{1,d-1}$-divergent (respectively $\Psf_{1,d-1}$-transverse). Moreover, in this case 
$$
\xi( \Lambda_{\theta}(\Gamma)) = \Lambda_{1,d-1}(\Phi(\Gamma)).
$$
\item If $\Gamma \subset \GG$ is discrete and $\Pc$ is a finite collection of subgroups in $\Gamma$, then $\Gamma \subset \GG$ is $\Psf_{\theta}$-Anosov relative to $\Pc$ if and only if $\Phi(\Gamma)$ is $\Psf_{1,d-1}$-Anosov relative to $\Pc^\prime: = \{ \Phi(P) : P \in \Pc\}$.

\end{enumerate}
\end{proposition}

\begin{remark} Part (6) is not explicitly stated in ~\cite[Prop.\ B.1]{CZZ3}, however it follows immediately from part (5) and the definitions. 
\end{remark} 

\section{Multiplicative estimates}

It is a general principle from linear algebra that if $U_\theta(A^{-1})$ is uniformly transverse to $U_\theta(B)$, then $\kappa_\theta(AB)$ is coarsely equal to 
$\kappa_\theta(A)+\kappa_\theta(B)$,
see, for example, \cite[Lem.\ A.7]{BPS}. We make use of two manifestations of this principle, the first in the context of transverse groups and the second in the context of relatively Anosov
groups.  Previous instances of this principle in our work include \cite[Lem.\ 6.2]{CZZ2} and \cite[Prop.\ 6.3]{CZZ3}.

In  the results of this section, $\norm{\cdot}$ denotes the norm on $\mathfrak a$ induced by the restriction of the Killing from to $\mathfrak a$. 

\begin{proposition}\label{prop:multiplicative estimate v1} Suppose $\Gamma \subset \GG$ is $\Psf_\theta$-transverse and $\d_{\Fc_\theta}$ is a distance on $\Fc_\theta$ which is induced by a Riemannian metric. For any $\epsilon > 0$ there exists $C =C(\epsilon) > 0$ such that: if $\gamma, \eta \in \Gamma$ and $\d_{\Fc_\theta}(U_\theta(\gamma^{-1}), U_\theta(\eta)) > \epsilon$, then 
$$
\norm{\kappa_\theta(\gamma\eta) - \kappa_\theta(\gamma)-\kappa_\theta(\eta)} \le C.
$$
\end{proposition}

\begin{proof}
We use the following special case of \cite[Lem.\ A.7]{BPS}.

\begin{lemma}\label{lem:multiplicative estimate linear case} Let $e_1,\dots, e_d$ denote the standard basis of $\Rb^d$. Suppose $g_1, g_2 \in \SL(d, \Rb)$ have singular value decomposition $g_1 =  m_1 a_1 \ell_1$ and $g_2 =  m_2 a_2 \ell_2$. Then 
$$
\sigma_1(g_1)\sigma_1(g_2)\sin(\theta) \le \sigma_1(g_1 g_2) \le  \sigma_1(g_1)\sigma_1(g_2)
$$
where $\theta: = \angle\left( m_2\ip{e_1}, \ell_1^{-1}\ip{e_2,\dots, e_d}\right)$ is the Euclidean angle between the subspaces $m_2\ip{e_1}$ and $\ell_1^{-1}\ip{e_2,\dots, e_d}$.
\end{lemma}

If the proposition fails, then for every $n \ge 1$ there exist $\gamma_n, \eta_n \in \Gamma$ where 
$$
\d_{\Fc_\theta}(U_\theta(\gamma_n^{-1}), U_\theta(\eta_n)) > \epsilon
$$
and 
$$
\norm{\kappa_\theta(\gamma_n\eta_n) - \kappa_\theta(\gamma_n)-\kappa_\theta(\eta_n)} \ge n.
$$ 
Lemma~\ref{lem:kappa multiplication estimate} implies that $\{\gamma_n\}$ and $\{\eta_n\}$ are escaping sequences in $\Gamma$. 
So by passing to a subsequence we can suppose that $U_\theta(\gamma_n^{-1}) \rightarrow F_1 \in \Lambda_\theta(\Gamma)$ and $U_\theta(\eta_n) \rightarrow F_2 \in \Lambda_\theta(\Gamma)$. 
Then $\d_{\Fc_\theta}(F_1, F_2) \ge \epsilon$ and hence $F_1$ and $F_2$ are transverse (since $\Gamma$ is $\Psf_\theta$-transverse).

Since $\{ \omega_\alpha|_{\mathfrak{a}_\theta} : \alpha \in \theta\}$ is a basis for $\mathfrak{a}_\theta^*$, after passing to a subsequence there exists some $\chi \in \sum_{\alpha \in \theta} \Nb \cdot \omega_\alpha$ such that 
$$
\lim_{n \rightarrow \infty} \abs{\chi\Big(\kappa_\theta(\gamma_n\eta_n) - \kappa_\theta(\gamma_n)-\kappa_\theta(\eta_n)\Big)} =\infty.
$$
Let $N \in \Nb$, $\Phi : \GG \rightarrow \PSL(d,\Rb)$ and $\xi : \Fc_\theta \rightarrow \Fc_{1,d-1}(\Rb^d)$ satisfy Proposition~\ref{prop:reduction to the linear case} for $\chi$. 
Then by Proposition~\ref{prop:reduction to the linear case} Claim (2), 
$$
\log \sigma_1(\Phi(g)) =N\chi(\kappa(g))=N\chi(\kappa_\theta(g))
$$
for all $g \in \GG$.

Let $\hat{\gamma}_n : = \Phi(\gamma_n)$ and $\hat{\eta}_n : =\Phi(\eta_n)$. Since $\Gamma$ is $\Psf_\theta$-transverse, $\alpha(\kappa(\gamma_n)) \rightarrow +\infty$ for all $\alpha \in \theta$. So by Proposition~\ref{prop:reduction to the linear case} Claim (4), 
$$
\lim_{n \rightarrow \infty} U_{1,d-1}(\hat{\gamma}_n ^{-1}) = \lim_{n \rightarrow \infty} \xi(U_\theta(\gamma_n^{-1})) = \xi(F_1).
$$
Likewise, $U_{1,d-1}(\hat{\eta}_n) \rightarrow \xi(F_2)$. Since $F_1$ and $F_2$ are transverse,  Proposition~\ref{prop:reduction to the linear case} Claim (1) implies that $\xi(F_1)$ and $\xi(F_2)$ are transverse. So Lemma~\ref{lem:multiplicative estimate linear case} implies that there exists a constant $C > 0$ such that 
$$
\abs{ \log \sigma_1(\hat{\gamma}_n\hat{\eta}_n) - \log \sigma_1(\hat{\gamma}_n) - \log \sigma_1(\hat{\eta}_n)} \le C
$$
for all $n \ge 1$. So,
$$
\abs{\chi\Big(\kappa_\theta(\gamma_n\eta_n)-\kappa_\theta(\gamma_n) -\kappa_\theta(\eta_n)\Big)} \le \frac{C}{N}
$$
and we have a contradiction. 
\end{proof} 

Our result for relatively Anosov groups involves a choice of Groves--Manning cusp space.

\begin{proposition}\label{prop:multiplicative estimate v2} Suppose $\Gamma \subset \GG$ is $\Psf_\theta$-Anosov relative to $\Pc$ and $X$ is a Groves--Manning cusp space for $(\Gamma, \Pc)$. There exists $C  > 0$ such that: if $f : [0,T] \rightarrow X$ is a geodesic with $f(0)=\id$ and $f(T) \in \Gamma$, then 
$$
\norm{\kappa_\theta(f(T)) - \kappa_\theta(f(t))-\kappa_\theta(f(t)^{-1}f(T))} \le C
$$
whenever $t \in [0,T]$ and $f(t) \in \Gamma$.
\end{proposition} 

\begin{proof}
Suppose not. Then for every $n \ge 1$ there is a geodesic $f_n : [0,T_n] \rightarrow X$ and some $t_n \in [0,T_n]$ such that $f_n(0)=\id$, $f_n(t_n) \in \Gamma$, $f_n(T_n)\in \Gamma$, and 
$$
\norm{\kappa_\theta(f_n(T_n)) - \kappa_\theta(f_n(t_n))-\kappa_\theta(f_n(t_n)^{-1}f_n(T_n))} \ge n.
$$
Let $\gamma_n := f(t_n)$ and $\eta_n := f_n(t_n)^{-1}f_n(T_n)$. By Lemma~\ref{lem:kappa multiplication estimate}, both $\{\gamma_n\}$ and $\{\eta_n\}$ are escaping sequences in $\Gamma$. So by passing to a subsequence we can suppose that $\gamma_n^{-1} \rightarrow x \in \partial_\infty X$ and $\eta_n \rightarrow y \in \partial_\infty X$. Since $t \mapsto f_n(t_n)^{-1} f_n(t)$ is a geodesic in $X$ passing through $\id$ and joining $\gamma_n^{-1}$ to $\eta_n$, we must have $x \neq y$.

Let $\xi : \partial_\infty X \rightarrow \Lambda_\theta(\Gamma)$ be the limit map. Proposition~\ref{prop:compactifications are the same} implies that $U_\theta(\gamma_n^{-1}) \rightarrow \xi(x)$ and $U_\theta(\eta_n) \rightarrow \xi(y)$. Since $\xi(x) \neq \xi(y)$, Proposition~\ref{prop:multiplicative estimate v1} implies that there exists $C > 0$ such that 
\begin{align*}
\norm{\kappa_\theta(f_n(T_n)) - \kappa_\theta(f_n(t_n))-\kappa_\theta(f_n(t_n)^{-1}f_n(T_n))} =\norm{\kappa_\theta(\gamma_n\eta_n) - \kappa_\theta(\gamma_n)-\kappa_\theta(\eta_n)} \le C
\end{align*} 
for all $n \ge 1$. So we have a contradiction. 
\end{proof}

\section{Resolution of singularities}\label{sec: resolution of singularities} 

In this section we study the asymptotic behavior of proper positive functions which are products of powers of rational functions. This is the key technical step needed to prove that the Poincar\'e series of a peripheral subgroup of a relatively Anosov subgroup diverges at its critical exponent, see Theorem \ref{thm:entropy gap for rel Anosov}.

Let $\lambda$ denote the Lebesgue measure on $\Rb^d$, and $\norm{\cdot}$ the standard Euclidean norm on $\Rb^d$.  We say that $R$ is a rational function on $\Rb^d$ if $R = \frac{f}{g}$ where $f,g : \Rb^d \rightarrow \Rb$ are polynomials and $R$ has domain $\{ g \neq 0\}$.

The main result of this section is the following theorem. Its proof is motivated by arguments of Benoist--Oh~\cite[Prop.\ 7.2]{benoist-oh}, which implies the special case when $m=1$. 

\begin{theorem}\label{thm:critical exponent of products of rational functions} 
Suppose $R_1,\dots, R_m$ are rational functions on $\Rb^d$ which are positive and everywhere defined, and let 
$$
R := R_1^{\ell_1} \cdots R_m^{\ell_m}
$$
where $\ell_1,\dots, \ell_m\in\Rb$. If $R$ is a proper function, then:
\begin{enumerate}
\item There exists $\delta=\delta(R) > 0$ such that 
$$
\int_{\Rb^d} R^{-s} d\lambda
$$
converges when $s \in (\delta,+\infty)$ and diverges when $s \in [0,\delta]$.
\item There exist $c=c(R)>0$ and $\epsilon=\epsilon(R)>0$ such that 
$$
R(x) \ge c(1+\norm{x})^\epsilon
$$
for all $x \in \Rb^d$. 
\end{enumerate}
\end{theorem}

To prove Theorem \ref{thm:critical exponent of products of rational functions} , we compactify $\Rb^d$ by identifying it with the affine subspace
$$
\Ab_1:=\{ [x_0:\cdots :x_d ] : x_0 \neq 0\}\subset\Pb(\Rb^{d+1})$$ 
via the coordinate chart $\psi_1:\Ab_1\to\Rb^d$ given by
$$\psi_1:[1: x_1 : \cdots :x_d]\mapsto \left(x_1,\dots, x_d\right).$$
It now suffices to prove the following lemma, which is the analog of Theorem \ref{thm:critical exponent of products of rational functions} on ``neighborhoods of infinity.''

\begin{lemma} \label{lem: resolution}
For each $p \in \Pb(\Rb^{d+1}) - \Ab_1$, there exist $\delta_p, c_p, \epsilon_p > 0$ and an open neighborhood $\Oc_p\subset \Pb(\Rb^{d+1})$ of $p$ such that:
\begin{enumerate}
\item The integral 
$$
\int_{\psi_1(\Oc_p \cap \Ab_1)} R^{-s} d \lambda 
$$
converges when $s \in (\delta_p, +\infty)$ and diverges when $s \in [0,\delta_p]$. 
\item $R(x) \ge c_p(1+ \norm{x})^{\epsilon_p}$ for all $x \in \psi_1(\Oc_p \cap \Ab_1)\subset\Rb^d$. 
\end{enumerate}
\end{lemma} 

Assuming Lemma \ref{lem: resolution}, we prove Theorem~\ref{thm:critical exponent of products of rational functions}. 

\begin{proof}[Proof of Theorem~\ref{thm:critical exponent of products of rational functions}]  
Since $ \Pb(\Rb^{d+1}) - \Ab_1$ is compact, there exist finitely many points $p_1,\dots, p_n$ in $ \Pb(\Rb^{d+1}) - \Ab_1$ such that 
$$
 \Pb(\Rb^{d+1}) - \Ab_1 \subset \bigcup_{j=1}^n \Oc_{p_j}.
 $$

Proof of (1). Let $\delta := \max\{ \delta_{p_j} : 1 \le j \le n\}$.  Since $R$ is positive,
\begin{align*}
\max_{1\le j \le n}\int_{\psi_1(\Oc_{p_j}\cap\Ab_1)}R^{-s} d\lambda&\le\int_{\Rb^d} R^{-s} d\lambda\le \int_{\psi_1(\Ab_1-  \bigcup_{j=1}^n \Oc_{p_j})} R^{-s} d\lambda+ \sum_{j=1}^n \int_{\psi_1(\Oc_{p_j}\cap\Ab_1)}R^{-s} d\lambda.
\end{align*}
Since $ \Ab_1 -  \bigcup_{j=1}^n \Oc_{p_j}$ is compact, the integral $\int_{\psi_1(\Ab_1-  \bigcup_{j=1}^n \Oc_{p_j})} R^{-s} d\lambda$ is finite, so Claim (1) of Lemma \ref{lem: resolution} implies that (1) holds. 

Proof of (2).  Let $\epsilon :=\min\{ \epsilon_{p_j} : 1 \le j \le n\}$. Since $R$ is positive and continuous on the compact set $K:=\psi_1(\Ab_1 -  \bigcup_{j=1}^n \Oc_{p_j})\subset\Rb^d$,  there exists $c_0 > 0$ such that 
\[
R(x) \ge c_0(1+\norm{x})^{\epsilon}
\] 
for all $x \in  K$. Thus, if we set $c:= \min\left( \{ c_0\} \cup \{ c_{p_j} : 1 \le j \le n\} \right)$, then Claim (2) of Lemma \ref{lem: resolution} implies that
\[R(x)\ge c(1+ \norm{x})^\epsilon\]
for all $x\in\Rb^d$.
\end{proof} 

To prove Lemma \ref{lem: resolution}, first note that we can assume that $p=[0:1:0:\dots:0]$ by changing coordinates. Then $p$ lies in the affine subspace
$$
\Ab_2:=\left\{ [x_0:\cdots :x_d ] : x_1 \neq 0\right\}\subset\Pb(\Rb^{d+1}).
$$
Let $\psi_2:\Ab_2\to\Rb^d$ be the coordinate chart given by
\[\psi_2:[y_1 : 1 : y_2 :\dots : y_d]\mapsto (y_1,\dots,y_d).\] 
Observe that $\psi_2(p)=0$ and 
\begin{align}\label{eqn: Zcdef}
\psi_2(\Ab_2 - \Ab_1)=\Zc:=\left\{(y_1,\dots,y_d)\in\Rb^d:y_1= 0\right\}.
\end{align}
Thus, the restriction 
\[\psi_1\circ\psi_2^{-1}|_{\Rb^d-\Zc}:\Rb^d-\Zc\to\Rb^d\]
is a well-defined embedding. 

Since $\psi_1\circ\psi_2^{-1}|_{\Rb^d-\Zc}$ is given by 
\begin{align}\label{eqn: Psi formula}
\psi_1\circ\psi_2^{-1}(y_1,\dots,y_d)=\left(\frac{1}{y_1},\frac{y_2}{y_1},\dots,\frac{y_d}{y_1}\right),
\end{align}  
for all $j \in \{1,\dots,m\}$,
\[T_j:=R_j \circ \psi_1\circ\psi_2^{-1}:\Rb^d-\Zc\to\Rb,\]
is a rational function that is well-defined and positive on $\Rb^d-\Zc$, and the assumption that $R$ is proper implies that $T:=T_1^{\ell_1}\dots T_m^{\ell_m}$ satisfies $\lim_{y\to z}T(y)=\infty$ for all $z\in\mathcal Z$. Also, \eqref{eqn: Psi formula} implies that the Jacobian $D(\psi_1\circ\psi_2^{-1})$ of $\psi_1\circ\psi_2^{-1}$ satisfies
\[\abs{\det D(\psi_1\circ\psi_2^{-1})}=\frac{1}{\abs{y_1}^{d+1}},\]
and
\[\norm{\psi_1\circ\psi_2^{-1}(y_1,\ldots,y_d)}=\sqrt{\frac{1+y_2^2+\dots+y_d^2}{y_1^2}}.\]
Thus, to prove Lemma \ref{lem: resolution}, it now suffices to prove the following lemma.

\begin{lemma} \label{lem: resolution 2}
Suppose that $T_1,\dots,T_m$ are rational functions on $\Rb^d$ which are positive and defined on $\Rb^d - \Zc$, and let 
\[T:=T_1^{\ell_1}\dots T_m^{\ell_m}\] 
for some $\ell_1,\dots,\ell_m\in\Rb$. If $\lim_{y\to z}T(y)=\infty$ for all $z\in\mathcal Z$, then there exist $\delta, c, \epsilon > 0$ and an open neighborhood $\Oc\subset \Rb^d$ of $0$ such that:
\begin{enumerate}
\item The integral 
$$
\int_{\Oc-\Zc} \frac{T^{-s}}{\abs{y_1}^{d+1}} \d \lambda 
$$
converges when $s \in (\delta, +\infty)$ and diverges when $s \in [0,\delta]$. 
\item For all $y=(y_1,\dots,y_d) \in \Oc-\Zc$, we have
\[T(y) \ge c\left(1+ \sqrt{\frac{1+y_2^2+\dots+y_d^2}{y_1^2}}\right)^{\epsilon}.\] 
\end{enumerate}
\end{lemma} 

For the remainder of this section, we will focus on the proof of Lemma \ref{lem: resolution 2}. An important tool used in the proof is the following version of Hironaka's theorem~\cite{MR0199184} on the resolution of singularities (as stated in~\cite[pg.\ 147]{MR256156}). 

\begin{theorem}[Resolution theorem]\label{thm: res} 
Let $F$ be a real analytic function defined in a neighborhood of $0 \in \Rb^d$, and let $\Zc$ be the set of zeroes of $F$. If $F$ is not identically zero, then there exists a neighborhood $U$ of $0$ in $\Rb^d$,  a real analytic manifold $M$ and a proper real analytic map $\Phi : M \rightarrow U$ such that 
\begin{enumerate} 
\item $\Phi$ restricts to a real-analytic diffeomorphism 
\[M - \Phi^{-1}(\Zc) \rightarrow U - \Zc.\]
\item For every $q \in M$ there exists an open neighborhood $V_q\subset M$ of $q$ and real analytic local coordinates $z_{q,1},\dots,z_{q,d}$ on $V_q$ centered at $q$ where
$$
F \circ \Phi(z_{q,1},\dots,z_{q,d}) = z_{q,1}^{k_{q,1}} \cdots z_{q,d}^{k_{q,d}} \cdot \hat F_q(z_{q,1},\dots, z_{q,d})
$$
for some $k_{q,1},\dots, k_{q,d} \in \Zb_{\ge 0}$ and some nowhere-vanishing real analytic function $\hat F_q:V_q\to\Rb$.
\end{enumerate} 
\end{theorem}

The main idea of the proof of Claim (1) of Lemma \ref{lem: resolution 2} is to try to factor the integrand $\frac{T^{-s}}{\abs{y_1}^{d+1}}$ into the product of a bounded, nowhere vanishing function on $\Oc$, and  power functions on $\Oc-\Zc$ that are ``responsible" for how $\frac{T^{-s}}{\abs{y_1}^{d+1}}$ goes to infinity or zero near $\Zc$. Doing so allows us to compare the required integral with the integral of a product of power functions, whose convergence or lack thereof is well-understood. Unfortunately, such a global factorization of $\frac{T^{-s}}{\abs{y_1}^{d+1}}$ on $\Oc$ is not possible in general. However, Theorem \ref{thm: res} ensures that we do have such a factorization locally, but at the cost of pre-composing the integrand with a given real analytic function $\Phi$. We will show that the complications introduced to this strategy by using $\Phi$ are surmountable, and that it can indeed be used to prove Lemma \ref{lem: resolution 2}.

We first write, for each $j\in\{1,\dots,n\}$, the rational function $T_j$ as
\[T_j= \frac{f_j}{g_j}\] 
where $f_j,g_j:\Rb^d\to\Rb$ are polynomials whose zeroes lie in the set $\Zc$ defined by \eqref{eqn: Zcdef}. Then
\begin{align}\label{eqn: F}
F:=f_1 \cdots f_m \cdot g_1 \cdots g_m\cdot y_1:\Rb^d\to\Rb,
\end{align}
is a polynomial whose set of zeroes is $\Zc$. Applying Theorem \ref{thm: res} to this $F$, we get  an open neighborhood $U$ of $0\in\mathbb R^d$, a real analytic manifold $M$,
a proper real analytic map \hbox{$\Phi:M\to U$} such that for all $q\in M$, there exists an open neighborhood $V_q$ of $q$ with local coordinates   $\{z_{q,r}\}_{r=1}^d$ centered at $q$ such that 
$F\circ \Phi|_{V_q}=z_{q,1}^{k_{q,1}} \cdots z_{q,d}^{k_{q,d}} \cdot \hat F_q$ and $\hat F_q:V_q\to\mathbb R$ is a nowhere-vanishing real analytic function.
Notice that, via the local coordinates $(z_{q,1},\dots,z_{q,d})$ on $V_q$, the Lebesgue measure $\lambda$ on $\Rb^d$ induces a measure $\lambda_q$ on $V_q$.

Fix a neighborhood $\Oc\subset\Rb^d$ of $0$ whose closure is compact and lies in $U$. Also, for each $q \in M$, fix an open neighborhood $V_q^\prime$ of $q$ whose closure lies in $V_q$. Since $\Phi$ is proper, we can find finitely many points $q_1, \dots, q_n \in M$ such that 
\begin{align*}
\Phi^{-1}(\Oc) \subset \bigcup_{i=1}^n  V_{q_i}^\prime. 
\end{align*}
For all $i\in\{1,\dots,n\}$ and $r\in\{1,\dots,d\}$, set 
\[V_i:=V_{q_i}^\prime \cap \Phi^{-1}(\Oc)\subset V_{q_i},\quad\hat F_i:=\hat F_{q_i}: V_{q_i} \rightarrow \Rb, \quad z_{i,r}:=z_{q_i,r}, \quad k_{i,r}:=k_{q_i,r}\quad \text{and}\quad\lambda_i:=\lambda_{q_i}.\] 
Notice that by construction, 
\begin{equation}\label{eqn:Vi is compact in Vqi}
\overline{V_i} \subset V_{q_i}
\end{equation}
is compact, and
\begin{align}\label{eqn: cover}
\Oc = \bigcup_{i=1}^n  \Phi (V_i). 
\end{align}

\begin{remark}
In the above set up, one might be tempted to say that by shrinking $\Oc$, one can find some $q\in M$ such that $\Oc\subset\Phi(V_q')$. However, this might not be possible: Each $\Phi(V_q)\subset U$ might not contain any open neighborhoods of $\Phi(q)$ even if $V_q\subset M$ is an open neighborhood of $q$.
\end{remark}

Using the sets $V_i$ and the map $\Phi$, we have the following local criterion for when the integral
$$
\int_{\Oc-\Zc} \frac{T^{-s}}{\abs{y_1}^{d+1}} \d \lambda 
$$
converges.

\begin{lemma}\label{lem: integral simpify}
For any $s\ge 0$, the integral 
\[
\int_{\Oc-\Zc} \frac{T^{-s}}{\abs{y_1}^{d+1}} \d \lambda 
\] converges if and only if the integral
\[\int_{V_i- \Phi^{-1}(\Zc)} \frac{(T \circ\Phi)^{-s} }{\abs{y_1 \circ \Phi}^{d+1}} \abs{\det D(\Phi) } {\rm d} \lambda_i\]
converges for all $i\in\{1,\dots,n\}$. Here, $D(\Phi)$ is the Jacobian of $\Phi$ restricted to $V_{q_i}$, with respect to the local coordinates $(z_{i,1},\dots,z_{i,d})$ of $V_{q_i}$.
\end{lemma}

\begin{proof}
By \eqref{eqn: cover},
\begin{align*}
\max_{i=1,\dots,n}\int_{\Phi(V_i)-\Zc} \frac{T^{-s}}{\abs{y_1}^{d+1}} {\rm d} \lambda
\le \int_{\Oc-\Zc} \frac{T^{-s}}{\abs{y_1}^{d+1}} \d \lambda  \le\sum_{i=1}^n \int_{\Phi(V_i)-\Zc} \frac{T^{-s}}{\abs{y_1}^{d+1}} {\rm d} \lambda.
\end{align*}
Since $\Phi$ restricts to a diffeomorphism $M - \Phi^{-1}(\Zc) \rightarrow U- \Zc$, we have
\[\int_{V_i- \Phi^{-1}(\Zc)} \frac{(T\circ\Phi)^{-s} }{\abs{y_1 \circ \Phi}^{d+1}} \abs{\det D(\Phi) } {\rm d} \lambda_i=\int_{\Phi(V_i)-\Zc} \frac{T^{-s}}{\abs{y_1}^{d+1}} {\rm d} \lambda\]
for each $i\in\{1,\dots,n\}$. The lemma follows.
\end{proof}

In light of Lemma \ref{lem: integral simpify}, we now need to understand for which values of $s\ge 0$  the integral
\[\int_{V_i- \Phi^{-1}(\Zc)} \frac{(T \circ\Phi)^{-s} }{\abs{y_1 \circ \Phi}^{d+1}} \abs{\det D(\Phi) } {\rm d} \lambda_i\]
converges. We do so using the local expressions of the real analytic functions $f_j\circ\Phi|_{V_{q_i}}$, $g_j \circ\Phi|_{V_{q_i}}$, $y_1\circ\Phi|_{V_{q_i}}$ and $\det D(\Phi)|_{V_{q_i}}$. By factoring their Taylor series in the local coordinates $(z_{i,1},\dots,z_{i,d})$ on $V_{q_i}$, we can write 
\begin{itemize}
\item $f_j\circ\Phi |_{V_{q_i}} = z_{i,1}^{a_{i,j,1}} \cdots z_{i,d}^{a_{i,j,d}} \cdot \hat{f}_{i,j}$ where $a_{i,j,1}, \dots, a_{i,j,d} \in \Zb_{\ge 0}$ and $\hat{f}_{i,j}:V_{q_i}\to\Rb$ is not identically zero on
\[\Zc_{i,r}:=\{(z_{i,1},\dots,z_{i,d})\in V_{q_i}: z_{i,r} = 0\}\] 
for any $r\in\{1,\dots,d\}$. 
\item $g_j\circ\Phi |_{V_{q_i}}= z_{i,1}^{b_{i,j,1}} \cdots z_{i,d}^{b_{i,j,d}} \cdot \hat{g}_{i,j}$ where  $b_{i,j,1}, \dots, b_{i,j,d} \in \Zb_{\ge 0}$ and $\hat{g}_{i,j}:V_{q_i}\to\Rb$ is not identically zero on $\Zc_{i,r}$ for any $r\in\{1,\dots,d\}$. 
\item $y_1\circ\Phi|_{V_{q_i}}= z_{i,1}^{c_{i,1}} \cdots z_{i,d}^{c_{i,d}} \cdot \hat{h}_i$ where $c_{i,1}, \dots, c_{i,d} \in \Zb_{\ge 0}$ and $\hat{h}_i:V_{q_i}\to\Rb$ is not identically zero on $\Zc_{i,r}$ for any $r\in\{1,\dots,d\}$. 
\item $\det D(\Phi)|_{V_{q_i}} = z_{i,1}^{\gamma_{i,1}} \cdots z_{i,d}^{\gamma_{i,d}} \cdot \hat{J}_i$, where $\gamma_{i,1},\dots,\gamma_{i,d} \in \Zb_{\ge 0}$ and $\hat{J}_i:V_{q_i}\to\Rb$ is not identically zero on $\Zc_{i,r}$ for any $r\in\{1,\dots,d\}$. 
\end{itemize}
Using Theorem \ref{thm: res}, we deduce the following lemma about the functions $\hat{f}_{i,j}$, $\hat{g}_{i,j}$, $\hat{h}_i$ and
$$
\hat{W}_{i,s} :=  \left(\frac{ \hat{g}_{i,1}^{\ell_1} \cdots \hat{g}_{i,m}^{\ell_m}}{\hat{f}_{i,1}^{\ell_1} \cdots \hat{f}_{i,m}^{\ell_m}}\right)^s  \frac{\hat{J}_i}{\hat{h}_i^{d+1}}
$$
for all $s\ge 0$. 

\begin{lemma}\label{lem: resolution properties 0}Fix $i\in\{1,\dots,n\}$.
\begin{enumerate}
\item $\hat F_i=\hat{f}_{i,1} \cdots \hat{f}_{i,m} \cdot \hat{g}_{i,1} \cdots \hat{g}_{i,m} \cdot \hat{h}_i$. In particular, $\hat{f}_{i,1}$, $\dots$, $\hat{f}_{i,m}$, $\hat{g}_{i,1}$, $\dots$, $\hat{g}_{i,m}$ and $\hat{h}_i$ are all nowhere-vanishing on $V_{q_i}$.
\item For any $s \ge 0$, $\hat{W}_{i,s}$ is a real analytic function which is bounded on $V_i$ and not identically zero on $\Zc_{i,r}$ for any $r\in\{1,\dots,d\}$. 

\end{enumerate}
\end{lemma}

\begin{proof} Proof of (1). Observe that
\begin{align*}
\frac{ \hat{f}_{i,1} \cdots \hat{f}_{i,m} \cdot \hat{g}_{i,1} \cdots \hat{g}_{i,m} \cdot \hat{h}_i}{ \hat{F}_i} = \prod_{r=1}^dz_{i,r}^{ k_{i,r} - c_{i,r} - \sum_{j=1}^m (a_{i,j,r} + b_{i,j,r}) }.
\end{align*}
By Theorem \ref{thm: res}, $\hat F_i$ is nowhere vanishing, so the left hand side is finite at every point in $V_{q_i}$. Thus 
$$
 k_{i,r} - c_{i,r} - \sum_{j=1}^m (a_{i,j,r} + b_{i,j,r}) \ge 0
 $$
 for all $r\in\{1,\dots,d\}$. Also, by the definition of $\hat f_{i,j}$, $\hat g_{i,j}$ and $\hat h_i$, the left hand side is not identically zero on $\Zc_{i,r}$ for all $r\in\{1,\dots,d\}$. Hence we must also have 
$$
 k_{i,r} - c_{i,r} - \sum_{j=1}^m (a_{i,j,r} + b_{i,j,r}) \le 0
 $$
 for all $r\in\{1,\dots,d\}$. So $\hat F_i=\hat{f}_{i,1} \cdots \hat{f}_{i,m} \cdot \hat{g}_{i,1} \cdots \hat{g}_{i,m} \cdot \hat{h}_i$. 
 
 Proof of (2). By (1) and the definition of $\hat{J}_i$, the function $\hat{W}_{i,s}$ is a real analytic function on $V_{q_i}$ which is not identically zero on $\Zc_{i,r}$ for any 
 $r\in\{1,\dots,d\}$. Since $\overline{V_i} \subset V_{q_i}$ is a compact subset (see Equation~\eqref{eqn:Vi is compact in Vqi}), the function $\hat{W}_{i,s}$ is bounded on $V_i$. 
\end{proof}

For all $i\in\{1,\dots,n\}$ and $r\in\{1,\dots,d\}$, set 
\[\beta_{i,r} := \sum_{j=1}^m \ell_j( b_{i,j,r} - a_{i,j,r}).\]
Since $T$ is positive on $\Rb^d-\Zc$, we can take the absolute value of each term to conclude that 
\begin{align}\label{eqn: integrand local expression}
\frac{(T\circ\Phi)^{-s} }{\abs{y_1 \circ \Phi}^{d+1}} \abs{\det D(\Phi) }&= \abs{\hat{W}_{i,s}}\prod_{r=1}^d|z_{i,r}|^{s\beta_{i,r}-(d+1)c_{i,r}+\gamma_{i,r}}.
\end{align}

The next lemma characterizes when integrals of functions with the above form converge. 

\begin{lemma}\label{lem:simple convergence criteria}  Suppose $i \in \{1,\dots, n\}$ and $W :  V_i \rightarrow \Rb$ is a bounded real analytic function which is not identically zero on the hyperplane $\Zc_{i,r}$ for any $r\in\{1,\dots,d\}$. Then for $\eta_1, \dots, \eta_d \in \Rb$ the integral 
$$
\int_{V_i- \Phi^{-1}(\Zc)} \abs{z_{i,1}}^{\eta_1} \cdots \abs{z_{i,d}}^{\eta_d} \abs{W} {\rm d} \lambda_i
$$
converges if and only if $\eta_r > -1$ for all $r \in \{1,\dots, d\}$. In particular,
\[\int_{\Oc-\Zc} \frac{T^{-s}}{\abs{y_1}^{d+1}} \d \lambda \]
converges if and only if $s\beta_{i,r}-(d+1)c_{i,r}+\gamma_{i,r}>-1$ for all $i\in\{1,\dots,n\}$ and all $r\in\{1,\dots,d\}$.
\end{lemma} 

\begin{proof} Since $W$ is bounded, it is clear that if $\eta_r > -1$ for all $r \in \{1,\dots, d\}$, then the integral converges. 

For the other direction, it suffices to assume that  $\eta_1 \le -1$ and then show that the integral diverges. Since $W$ is not identically zero on $\Zc_{i,1}$, we can find an open set $V^\prime \subset \Zc_{i,1}$ such that $V^\prime \subset V_i$ and $W$ is nowhere vanishing on $V^\prime$. 
By further shrinking $V^\prime$ we can assume that there exists  $\epsilon > 0$ such that $\overline{(0,\epsilon) \times V^\prime} \subset V_i$ and 
$$
\overline{V^\prime} \cap \bigcup_{r=2}^d \Zc_{i,r} = \emptyset.
$$
Then
$$
\alpha: = \min_{ z \in \overline{(0,\epsilon) \times V^\prime}} \abs{z_{i,2}}^{\eta_2} \cdots \abs{z_{i,d}}^{\eta_d} \abs{W}
$$
is positive. Hence, 
\begin{align*}
\int_{V_i- \Phi^{-1}(\Zc)}&  \abs{z_{i,1}}^{\eta_1} \cdots \abs{z_{i,d}}^{\eta_d} \abs{W} {\rm d} \lambda_i \ge
\int_{(0,\epsilon) \times V^\prime} \abs{z_{i,1}}^{\eta_1} \cdots \abs{z_{i,d}}^{\eta_d} \abs{W} {\rm d} \lambda_i \\
&  \ge \alpha {\rm Vol}(V^\prime) \int_0^\epsilon t^{\eta_1} dt = + \infty
\end{align*}
where ${\rm Vol}$ is the measure on $\Zc_{i,1}$ induced by the coordinates $(z_{q,2},\dots,z_{q,d})$ on $\Zc_{i,1}$ and the Lebesgue measure on $\Rb^{d-1}$. 

To prove the second statement of the lemma, first note that by Lemma \ref{lem: resolution properties 0} Claim (2), $\hat{W}_{i,s}$ is a bounded real analytic function on $V_i$ 
and not identically zero on the hyperplane $\Zc_{i,r}$ for any $r\in\{1,\dots,d\}$. So, we may apply the first statement of the lemma to 
Equation \eqref{eqn: integrand local expression} and deduce that
\[\int_{V_i- \Phi^{-1}(\Zc)} \frac{(T\circ \Phi)^{-s} }{\abs{y_1 \circ \Phi}^{d+1}} \abs{\det D(\varphi_i) }  {\rm d} \lambda_i\]
converges if and only if 
$$
s \beta_{i,r}- (d+1)c_{i,r} + \gamma_{i,r} > -1
$$
for all $r\in\{1,\dots,d\}$. The second statement of the lemma now follows from Lemma \ref{lem: integral simpify}.
\end{proof} 

By Lemma \ref{lem:simple convergence criteria}, we now need to know the values of $s$ so that $s\beta_{i,r}-(d+1)c_{i,r}+\gamma_{i,r}>-1$ for all $i\in\{1,\dots,n\}$ and all $r\in\{1,\dots,d\}$. To that end, it is useful to have the following relations between the exponents $\beta_{i,r}$, $c_{i,r}$ and $\gamma_{i,r}$.

\begin{lemma}\label{lem: resolution properties}\
\begin{enumerate}
\item For all $i\in\{1,\dots,n\}$ and all $r\in\{1,\dots,d\}$, $\beta_{i,r}\ge 0$.
\item For all $i\in\{1,\dots,n\}$, $\{ r : c_{i,r} > 0\}\subset\{ r : \beta_{i,r} > 0\}$. 
\item There exist $i\in\{1,\dots,n\}$ and $r\in\{1,\dots,d\}$ such that 
\[c_{i,r} - \gamma_{i,r}\ge 1.\]
\end{enumerate}
In particular, 
$$
\delta : = \max\left\{ \frac{(d+1)c_{i,r}-\gamma_{i,r}-1}{\beta_{i,r}} : i\in\{1,\dots,n\}, \, r\in\{1,\dots,d\}\text{ such that }\beta_{i,r} > 0 \right\}
$$
is a positive real number. 
\end{lemma}

\begin{proof} Proof of (1). 
Since $T$ is positive on $U-\Zc$ and $\lim_{y\to p}T(y)=\infty$ for all $p\in \Zc$, the function 
\[\frac{1}{T\circ\Phi}:M-\Phi^{-1}(\Zc)\to\Rb\] 
extends to a continuous function on all of $M$. 
In particular, for all $i\in\{1,\dots,n\}$,
\begin{equation}\label{eqn:expression for R} 
\frac{1}{T\circ\Phi}\bigg|_{V_i} = z_{i,1}^{\beta_{i,1}} \cdots z_{i,d}^{\beta_{i,d}}  \frac{ \hat{g}_{i,1}^{\ell_1} \cdots \hat{g}_{i,m}^{\ell_m}}{\hat{f}_{i,1}^{\ell_1} \cdots \hat{f}_{i,m}^{\ell_m}}
\end{equation} 
is finite. By Claim (1) of Lemma \ref{lem: resolution properties 0} and the fact that $\overline{V_i} \subset V_{q_i}$ is compact, we see that $\frac{ \hat{g}_{i,1}^{\ell_1} \cdots \hat{g}_{i,m}^{\ell_m}}{\hat{f}_{i,1}^{\ell_1} \cdots \hat{f}_{i,m}^{\ell_m}}$ is bounded and nowhere vanishing on $V_i$, so (1) follows. 

Proof of (2). Since $\lim_{y\to z}T(y)=\infty$ for all $z\in\mathcal Z$, it follows that $\frac{1}{T}$ vanishes on $\Zc$. Equivalently, $\frac{1}{T\circ\Phi}$ vanishes wherever $y_1 \circ\Phi$ vanishes. By Lemma \ref{lem: resolution properties 0} Claim (1) and the fact that $\overline{V_i} \subset V_{q_i}$ is compact, $\hat{h}_i$ and $\frac{ \hat{g}_{i,1}^{\ell_1} \cdots \hat{g}_{i,m}^{\ell_m}}{\hat{f}_{i,1}^{\ell_1} \cdots \hat{f}_{i,m}^{\ell_m}}$ are bounded and nowhere vanishing on $V_i$, so (2) follows from Equation~\eqref{eqn:expression for R}.

Proof of (3). Observe that the integral
\[\int_{\Oc-\Zc} \frac{1}{\abs{y_1}}{\rm d} \lambda\]
diverges. By \eqref{eqn: cover}, 
\[\int_{\Oc-\Zc} \frac{1}{\abs{y_1}}{\rm d} \lambda\le\sum_{i=1}^n \int_{V_i - \Phi^{-1}(\Zc)} \frac{\abs{\det D(\Phi)}}{|y_1\circ\Phi|} {\rm d} \lambda_i=\sum_{i=1}^n\int_{V_i- \Phi^{-1}(\Zc)} \left(\prod_{r=1}^d\abs{z_{i,r}}^{-c_{i,r} + \gamma_{i,r}}\right) \abs{\frac{\hat J_i}{\hat h_i}}  {\rm d} \lambda_i,\]
so there is some $i\in\{1,\dots,n\}$ such that the integral
\[\int_{V_i- \Phi^{-1}(\Zc)} \left(\prod_{r=1}^d\abs{z_{i,r}}^{-c_{i,r} + \gamma_{i,r}}\right) \abs{\frac{\hat J_i}{\hat h_i}}  {\rm d} \lambda_i\] 
diverges. By Lemma \ref{lem: resolution properties 0} Claim (1), $\frac{\hat J_i}{\hat h_i}$ is real analytic on $V_{q_i}$. Also, by the definition of $\hat{J}_i$, $\frac{\hat J_i}{\hat h_i}$ is not identically zero on $\Zc_{i,r}$ for all $r\in\{1,\dots,d\}$. Since $\overline{V_i} \subset V_{q_i}$ is compact, the function $\frac{\hat J_i}{\hat h_i}$ is bounded on $V_i$. So by Lemma~\ref{lem:simple convergence criteria}  there is some $r\in\{1,\dots,d\}$ such that $-c_{i,r} + \gamma_{i,r}\le -1$. This proves (3).

We will now deduce the final claim of the lemma. By (3), there is some $i\in\{1,\dots,n\}$ and $r\in\{1,\dots,d\}$ such that $c_{i,r} - \gamma_{i,r}\ge 1$. Since $\gamma_{i,r}\ge 0$, it follows that $c_{i,r}>0$, so (2) implies that $\beta_{i,r}>0$. Then
\[
\frac{(d+1)c_{i,r}-\gamma_{i,r}-1}{\beta_{i,r}}>0,
\]
which implies that $\delta>0$.
\end{proof}

Combining Lemmas~\ref{lem:simple convergence criteria} and  \ref{lem: resolution properties}, we may now prove Lemma \ref{lem: resolution 2}.

\begin{proof}[Proof of Lemma \ref{lem: resolution 2}]
Proof of (1). Let $\delta>0$ be the quantity specified in the statement of Lemma \ref{lem: resolution properties}. 

Suppose that $s \in (\delta,+\infty)$. Pick any $i\in\{1,\dots,n\}$ and $r\in\{1,\dots,d\}$. If $\beta_{i,r}=0$, then Lemma \ref{lem: resolution properties} Claim (2) implies that $c_{i,r}=0$, in which case 
\[s \beta_{i,r}- (d+1)c_{i,r} + \gamma_{i,r}=\gamma_{i,r} \ge 0 > -1.\]
If $\beta_{i,r}\neq  0$, then $s>\delta\ge \frac{(d+1)c_{i,r}-\gamma_{i,r}-1}{\beta_{i,r}}$. By Lemma \ref{lem: resolution properties} Claim (1), $\beta_{i,r}>0$, so
\[s \beta_{i,r}- (d+1)c_{i,r} + \gamma_{i,r}> -1.\]
It now follows from Lemma \ref{lem:simple convergence criteria} that
\[\int_{\Oc-\Zc} \frac{T^{-s}}{\abs{y_1}^{d+1}} \d \lambda \]  
converges. 

Next, suppose that $s \in[0,\delta]$. Then by definition, there is some $i\in\{1,\dots,n\}$ and $r\in\{1,\dots,d\}$ such that 
\[ s\beta_{i,r}- (d+1)c_{i,r} + \gamma_{i,r}\le -1.\]
Then by Lemma \ref{lem:simple convergence criteria},
\[\int_{\Oc-\Zc} \frac{T^{-s}}{\abs{y_1}^{d+1}} \d \lambda \] 
diverges. This completes the proof of (1).

Proof of (2). Recall that
$$\frac{1}{T\circ\Phi}\bigg|_{V_i} = z_{i,1}^{\beta_{i,1}} \cdots z_{i,d}^{\beta_{i,d}}  R$$
where $R=\frac{ \hat{g}_{i,1}^{\ell_1} \cdots \hat{g}_{i,m}^{\ell_m}}{\hat{f}_{i,1}^{\ell_1} \cdots \hat{f}_{i,m}^{\ell_m}}$  is a nowhere vanishing analytic function
(by Lemma \ref{lem: resolution properties 0}) and that
$$y_1\circ\Phi|_{V_{q_i}}= z_{i,1}^{c_{i,1}} \cdots z_{i,d}^{c_{i,d}} \cdot \hat{h}_i$$
where $ \hat h_i$ is a nowhere vanishing analytic function
(again by Lemma \ref{lem: resolution properties 0}).
Since $\overline{V_i} \subset V_{q_i}$ is compact, $\beta_{i,r}\ge 0$ for all $r$ and  $\beta_{i,r}=0$ whenever $c_{i,r}=0$ (by Lemma \ref{lem: resolution properties}),
there exist $C_i,\epsilon_i>0$ such that
\begin{align}\label{eqn: bound in Vi}
\frac{1}{T\circ\Phi(z)}\le C_i\abs{y_1\circ\Phi(z)}^{\epsilon_i}
\end{align}
for all $z\in V_i$. Set
\[\epsilon:=\max\{\epsilon_1,\dots,\epsilon_n\}.\]
Since the closure of $\Oc$ is compact, there exists $c_0 > 0$ such that 
$$
1 \ge c_0\left(\abs{y_1}+\sqrt{1+y_2^2+\dots+y_d^2}\right)
$$
for all $y \in \Oc$. Hence 
\begin{equation}\label{eqn: stupid bound on 1/y1}
\frac{1}{\abs{y_1}} \ge c_0\left(1+\sqrt{\frac{1+y_2^2+\dots+y_d^2}{y_1^2}}\right)
\end{equation}
for all $y \in \Oc - \Zc$. Set
\[c:=\min\left\{\frac{c_0^{\epsilon_1}}{C_1},\dots,\frac{c_0^{\epsilon_n}}{C_n}\right\}.\]

Fix $y\in \Oc-\Zc$. Then there exist $i\in\{1,\dots,m\}$ and $z\in V_i$ such that $y=\Phi(z)$. Then
\[T(y)\ge\frac{1}{C_i\abs{y_1}^{\epsilon_i}}\ge c\left(1+\sqrt{\frac{1+y_2^2+\dots+y_d^2}{y_1^2}}\right)^{\epsilon},\]
where the first inequality holds by Equation \eqref{eqn: bound in Vi} and the second inequality holds by Equation \eqref{eqn: stupid bound on 1/y1}.

\end{proof}

\section{Entropy gap for peripheral subgroups}\label{sec: relating to integral} 

In this section we prove that the Poincar\'e series associated to any peripheral subgroup diverges at its critical exponent. 

\begin{theorem}[Theorem \ref{peripheral divergent}]\label{thm:entropy gap for rel Anosov}
Suppose $\Gamma\subset \GG$ is a $\Psf_\theta$-Anosov subgroup relative to $\Pc$,  $\phi\in \mathfrak{a}_\theta^*$ and $\delta^\phi(\Gamma) < +\infty$. If $P \in \Pc$, then $Q_P^\phi$ diverges at its critical exponent. 
\end{theorem}  

Delaying the proof of Theorem~\ref{thm:entropy gap for rel Anosov} for a moment, we observe that it implies that the critical exponent of the
peripheral subgroup is strictly smaller than the critical exponent of the entire group.

\begin{corollary}\label{cor:entropy gap with peripherals} 
Suppose $\Gamma\subset \GG$ is a $\Psf_\theta$-Anosov subgroup relative to $\Pc$,  $\phi\in \mathfrak{a}_\theta^*$ and $\delta^\phi(\Gamma) < +\infty$. If $P \in \Pc$, then $\delta^\phi(P) < \delta^\phi(\Gamma)$. 
\end{corollary} 

\begin{proof} Notice that $\Lambda_\theta(P)$ consists of a single point, namely the fixed point of $P$ in $\Lambda_\theta(\Gamma)$. Hence Theorems~\ref{thm: entropy gap CZZ3} and~\ref{thm:entropy gap for rel Anosov} imply that $\delta^\phi(P) < \delta^\phi(\Gamma)$. 
\end{proof}

The rest of the section is devoted to the proof of Theorem~\ref{thm:entropy gap for rel Anosov}, so fix $\Gamma$, $\Pc$ and $\phi$ as in the statement of the theorem.

Let $p : \GG \rightarrow \GG^\prime$ and $\theta^\prime \subset \Delta^\prime$ be as in Proposition~\ref{prop:reducing to no simple factors}. By part (4) of that proposition, $\Gamma^\prime : = p(\Gamma)$ is a $\Psf^\prime_{\theta^\prime}$-Anosov subgroup relative to $\Pc^\prime := \{ p(P) : P \in \Pc\}$. Also let $\phi^\prime \in (\mfa_{\theta^\prime}^\prime)^*$ be the unique functional where $\phi^\prime \circ \d p = \phi$. Then $Q_{p(P)}^{\phi^\prime}=Q_P^\phi$ for all $P \in \Pc$. So by replacing $\GG$ with $\GG^\prime$, we may assume that $\GG$ has trivial center, and that $\Psf_\theta$ contains no simple factors of $\GG$.

Fix $P \in \Pc$. Then by Theorem~\ref{thm:properties of relatively Anosov representations} there exists a closed subgroup $\mathsf{H} \subset \GG$ with finitely many components such that: \begin{enumerate}
\item $P$ is a cocompact lattice in $\mathsf{H}$.
\item $\mathsf{H} = \mathsf{L} \ltimes \mathsf{U}$ where $\mathsf{L}$ is compact and $\mathsf{U}$ is the unipotent radical of $\mathsf{H}$. 
\item $\mathsf{H}^0 = \mathsf{L}^0 \times \mathsf{U}$ and $\mathsf{L}^0$ is Abelian. 
\end{enumerate} 
Let $\mathfrak{u}$ denote the Lie algebra of $\mathsf{U}$. 
 
Since $\omega_\alpha(\kappa_\theta(g))=\omega_\alpha(\kappa(g))$ for all $\alpha\in\theta$ and $g\in\GG$, by Proposition~\ref{prop:properties of unipotent subgroups}, for any $\alpha \in \theta$
there exist $M_\alpha \in \Nb$, $C_\alpha>1$ and a positive everywhere defined rational function $R_\alpha: \mathfrak{u} \rightarrow \Rb$ where 
$$
\frac{1}{C_\alpha} R_\alpha(Y)^{1/M_\alpha} \le e^{\omega_\alpha(\kappa_\theta(\exp(Y)))} \le C_\alpha R_\alpha(Y)^{1/M_\alpha}
$$
for all $Y \in \mathfrak{u}$. Write $\phi = \sum_{\alpha \in \theta} c_\alpha \omega_\alpha$. Then define $R : = \prod_{\alpha \in \theta} R_\alpha^{\abs{c_\alpha}/M_\alpha}$ and $C_\phi: = \prod_{\alpha \in \theta} C_\alpha^{\abs{c_\alpha}}$. Note that
\[C_\phi^{-s} R^{-s}(Y) \le e^{-s\phi(\kappa_\theta(\exp(Y)))} \le C_\phi^s R^{-s}(Y)\]
for all $s\in\Rb$. 

\begin{lemma}\label{lem: R is proper} $R : \mathfrak{u} \rightarrow \Rb$ is proper. \end{lemma} 

\begin{proof} Suppose that $\{Y_n\}$ is an escaping sequence in $\mathfrak{u}$. Since $\exp: \mathfrak{u} \rightarrow \mathsf{U}$ is a diffeomorphism, $\{\exp(Y_n)\}$ is an escaping sequence in $\mathsf{U}$ (see Proposition~\ref{prop:properties of unipotent subgroups}). Since $P$ is a cocompact lattice in $\mathsf{H}$, there exists an escaping sequence $\{\gamma_n\}$ in $P$ such that $\{ \gamma_n^{-1}\exp(Y_n)\}$ is relatively compact in $\mathsf{H}$. So there exists $C_0 > 0$ such that 
$$
\norm{\kappa\left( \gamma_n^{-1}\exp(Y_n)\right)} \le C_0
$$
for all $n \ge 1$. Then by Lemma~\ref{lem:kappa multiplication estimate} 
$$
R(Y_n) \ge \frac{1}{C_\phi} e^{\phi(\kappa_\theta(\exp(Y_n)))} \ge \frac{1}{C_\phi} e^{\phi(\kappa_\theta(\gamma_n))} e^{-\norm{\phi}C_0},
$$
where $\norm{\phi}$ is the operator norm of the linear map $\phi:\mathfrak a\to\Rb$. Since $\delta^\phi(\Gamma) < +\infty$ and $\{\gamma_n\}$ is an escaping sequence, we must have $\phi(\kappa_\theta(\gamma_n))\rightarrow +\infty$. Hence $R(Y_n) \rightarrow +\infty$. So $R$ is proper. 
\end{proof} 

Fix a measure $\lambda$ on $\mathfrak{u}$ which is obtained by pulling back the Lebesgue measure on $\Rb^{\dim \mathfrak{u}}$ via some linear isomorphism 
(notice that any two measures obtained this way are scalar multiplies of each other). 
Then the push-forward $\mu : = \exp_*\lambda$ is a Haar measure on $\mathsf{U}$, see Proposition~\ref{prop:properties of unipotent subgroups}. By Theorem~\ref{thm:critical exponent of products of rational functions} there exists $\delta > 0$ such that
$$
Q_R(s): =\int_{\mathfrak{u}} R^{-s} d\lambda
$$
converges when $s \in (\delta, +\infty)$ and diverges when $s \in [0, \delta]$. Hence to complete the proof of Theorem~\ref{thm:entropy gap for rel Anosov} it suffices to show the following. 

\begin{lemma} There exists a continuous function $A:\Rb_{\ge 0} \to\Rb_{>0}$ such that
$$
\frac{1}{A(s)} Q_P^\phi(s) \le Q_R(s)  \le A(s) Q_P^\phi(s)
$$
for all $s\ge 0$.
\end{lemma}

\begin{proof} We prove the lemma via a series of estimates. First let 
$$
P_0 : = \mathsf{H}^0 \cap P = \left(\mathsf{L}^0 \times \mathsf{U}\right) \cap P.
$$ 
Since $\mathsf{H}$ has finitely many connected components, $P_0$ has finite index in $P$. Let $\gamma_1,\dots,\gamma_n\in P$ such that $P/P_0=\{\gamma_1 P_0,\dots,\gamma_n P_0\}$, and let 
\[D:=\norm{\phi}\max_{i=1,\dots,n}\norm{\kappa(\gamma_i)},\] 
where $\norm{\phi}$ is the operator norm of the linear map $\phi:\mathfrak a_\theta \to\Rb$. Then by Lemma~\ref{lem:kappa multiplication estimate}, 
\begin{align}\label{eqn: Q estimates 1}
\frac{e^{-Ds}}{n} Q_{P}^{\phi}(s) \le Q_{P_0}^\phi(s)  \le Q_{P}^{\phi}(s)
\end{align}
for all $s\ge 0$. 

Next, let $\pi : \mathsf{L}^0 \times \mathsf{U} \rightarrow \mathsf{U}$ denote the projection and let $P_1:=\pi(P_0)$. Since $P_0$ is discrete and $\mathsf{L}^0$ is compact, the kernel of $\pi|_{P_0}$ is finite and $P_1$ is discrete. Then by Lemma~\ref{lem:kappa multiplication estimate}, 
\begin{align}\label{eqn: Q estimates 2}
\frac{e^{-Cs}}{B} Q_{P_0}^\phi(s) \le Q_{P_1}^\phi(s)  \le Q_{P_0}^\phi(s)
\end{align}
for all $s\ge 0$, where $B:=\#\ker(\pi|_{P_0})$ and
\[C:=\norm{\phi}\max_{g\in \mathsf{L}^0}\norm{\kappa(g)}.\]

By \eqref{eqn: Q estimates 1} and \eqref{eqn: Q estimates 2}, it now suffices to show that there exists a continuous function $A:\Rb_{\ge 0} \to\Rb_{>0}$ such that
\begin{align}\label{eqn: Q estimates 3}
\frac{1}{A(s)} Q_{P_1}^\phi(s) \le Q_R(s)  \le A(s) Q_{P_1}^\phi(s)
\end{align}
for all $s\ge 0$.

Since $P$ is a cocompact lattice in $\mathsf{H}$, we see that $P_1$ is a cocompact lattice in $\mathsf{U}$. Fix compact neighborhoods $K_1, K_2 \subset \mathsf{U}$ of the identity such that the left $P_1$-translates of $K_1$ are pairwise disjoint and the left $P_1$-translates of $K_2$ cover $\mathsf{U}$. Since $K_1$ and $ K_2$ have non-empty interior,
both $\mu(K_1)$ and $ \mu(K_2)$ are positive. For both $i=1,2$, define
\[E_i:=\norm{\phi}\max_{h\in K_i}\norm{\kappa(h)}.\]
Then by Lemma~\ref{lem:kappa multiplication estimate},
\[e^{-sE_i}e^{-s\phi(\kappa_\theta(g))}\le e^{-s\phi(\kappa_\theta(\gamma))}\le e^{sE_i}e^{-s\phi(\kappa_\theta(g))}\]
for all $\gamma\in P_1$, $g\in\gamma K_i$ and $s\ge 0$. So
$$
\frac{e^{-sE_2}}{\mu(K_2)} \int_{\gamma K_2} e^{-s\phi(\kappa_\theta(g))} d\mu(g) \le e^{-s\phi(\kappa_\theta(\gamma))} \le \frac{e^{sE_1}}{\mu(K_1)}  \int_{\gamma K_1} e^{-s\phi(\kappa_\theta(g))} d\mu(g)
$$
for all $\gamma \in P_1$ and $s> 0$. Then 
\begin{align*}
Q_R(s) & = \int_{\mathfrak{u}} R^{-s}(Y) d\lambda(Y) \le C_\phi^{s}  \int_{\mathfrak{u}} e^{-s\phi(\kappa_\theta(\exp(Y)))} d\lambda(Y)= C_\phi^{s}\int_{\mathsf{U}} e^{-s\phi(\kappa_\theta(g))} d\mu(g) \\
& \le C_\phi^{s} \sum_{\gamma \in P_1} \int_{\gamma K_2} e^{-s\phi(\kappa_\theta(g))} d\mu(g)\le C_\phi^{s} \mu(K_2)e^{sE_2} Q_{P_1}^\phi(s). 
\end{align*}
Likewise, 
\begin{align*}
Q_R(s) & \ge C_\phi^{-s} \mu(K_1)e^{-sE_1} Q_{P_1}^\phi(s), 
\end{align*}
so \eqref{eqn: Q estimates 3} holds.
\end{proof}

\begin{remark}
Notice that the proof implies that $\delta^\phi(P)$ depends only on $\phi$ and $\mathsf{U}$. One may concisely describe $\mathsf{U}$ as the unipotent radical of the Zariski
closure of $P$.
\end{remark}

\section{Divergence of the Poincar\'e series at its critical exponent}\label{sec: adapting DOP}

In this section we use an argument of Dal'bo--Otal--Peign\'e~\cite{DOP} and Corollary \ref{cor:entropy gap with peripherals}  to prove that the Poincar\'e series diverges at its critical exponent. 
A similar argument was used by Blayac--Zhu~\cite{BZ} in their study of  relatively Anosov subgroups of the projective general linear group which preserve a properly convex domain and  Patterson--Sullivan measures defined using the Busemann functions associated to the Hilbert distance on the properly convex domain.

\begin{theorem}[Theorem \ref{intro: rel anosov divergent}]
\label{rel anosov divergence}
If $\Gamma\subset \GG$ is a $\Psf_\theta$-Anosov subgroup relative to $\Pc$,  $\phi\in \mathfrak{a}_\theta^*$ and $\delta^\phi(\Gamma) < +\infty$, then $Q_\Gamma^\phi$ diverges at its critical exponent. 
\end{theorem}

The key result needed to prove Theorem \ref{rel anosov divergence} is the following lemma. 

\begin{lemma}\label{lem: parabolic point}
If $\Gamma\subset \GG$ is a $\Psf_\theta$-Anosov subgroup relative to $\Pc$,  $\phi\in \mathfrak{a}_\theta^*$ and $\delta^\phi(\Gamma) < +\infty$, then
there exists a $\phi$-Patterson--Sullivan measure $\mu$ for $\Gamma$ of dimension $\delta^\phi(\Gamma)$ such that $\mu$ has no atoms.  
\end{lemma}

Assuming Lemma~\ref{lem: parabolic point}, we prove Theorem \ref{rel anosov divergence}.

\begin{proof}[Proof of Theorem \ref{rel anosov divergence} assuming Lemma~\ref{lem: parabolic point}]
By Lemma~\ref{lem: parabolic point} there exists an atom-less $\phi$-Patterson--Sullivan measure $\mu$ of dimension $\delta^\phi(\Gamma)$. Suppose for a contradiction that $Q_\Gamma^\phi(\delta^\phi(\Gamma)) < +\infty$. Then Theorem~\ref{thm:dichotomy} implies that  $\mu(\Lambda_{\theta}^{\rm con}(\Gamma))=0$. Since $\Lambda_\theta(\Gamma) - \Lambda_\theta^{\rm con}(\Gamma)$ is countable and $\mu$ has no atoms, $\mu(\Lambda_\theta(\Gamma)) = 0$. Since $\mu$ is a probability measure on $\Lambda_\theta(\Gamma)$ this is impossible. 
\end{proof}

We now give the proof of Lemma \ref{lem: parabolic point}.

\begin{proof}[Proof of Lemma \ref{lem: parabolic point}]Let $\delta:=\delta^\phi(\Gamma)$. Endow $\Gamma \cup \Lambda_\theta(\Gamma)$ with the topology from Lemma~\ref{lem: limit set compactifies} and for $x \in \Gamma \cup \Lambda_\theta(\Gamma)$ let $\mathcal D_{x}$ denote the Dirac measure supported on $x$. By~\cite[Lem.\ 3.1]{patterson} there exists a continuous non-decreasing function $h:\Rb^+\to\Rb^+$ such that:
\begin{enumerate}
\item The series
\[
\hat{Q}(s):=\sum_{\gamma\in \Gamma}h\left(e^{\phi(\kappa_\theta(\gamma))}\right)e^{-s\phi(\kappa_\theta(\gamma))}
\]
converges for $s > \delta$ and diverges for $s \le \delta$.
\item For any $\epsilon > 0$ there exists $\nu_0> 0$ such that: if $s > 1$ and $\nu > \nu_0$, then $h(\nu s) \le s^{\epsilon} h(\nu)$. 
\end{enumerate} 
(In the case when $Q_\Gamma^\phi$ diverges at its critical exponent, we can choose $h \equiv 1$.) 

For $s > \delta$ consider the probability measure 
$$
\mu_s:=\frac{1}{\hat{Q}(s)}\sum_{\gamma\in\Gamma} h\left(e^{\phi(\kappa_\theta(\gamma))}\right)e^{-s\phi(\kappa_\theta(\gamma))}\mathcal D_{\gamma}
$$
on $\Gamma \cup \Lambda_\theta(\Gamma)$. By compactness,  there exists $\{s_n\}\subset (\delta,\infty)$ so that $\lim_{n \rightarrow \infty} s_n=\delta$ and
\[
\mu:=\lim_{n \rightarrow \infty} \mu_{s_n}
\]
exists in the weak-$*$. topology. In the proof of~\cite[Prop.\ 3.1]{CZZ3}, we showed that $\mu$ is a $\phi$-Patterson--Sullivan measure for $\Gamma$ of dimension $\delta$.

We will now prove that $\mu$ has no atoms. By~\cite[Prop.\ 8.1]{CZZ3}, if $F \in  \Lambda_\theta^{\rm con}(\Gamma)$, then $\mu(\{F\})=0$. So it suffices to fix a parabolic point $F_0 \in \Lambda_\theta(\Gamma)$ and show that $\mu(\{F_0\})=0$.  By the definition of the weak-$*$ topology, it suffices to find a sequence of open sets $\{V_m\}$ in $\Gamma\cup\Lambda_\theta(\Gamma)$ each of which contains $F_0$, and where
\begin{align}\label{eqn: no parabolic atoms}
\liminf_{m \rightarrow \infty} \, \limsup_{s \searrow \delta}\, \mu_s(V_m) =0.
\end{align}

Let $\xi:\partial(\Gamma,\Pc)\to\Fc_\theta$ be the limit map of $\Gamma$. Let $P$ be the stabilizer in $\Gamma$ of $F_0$ (it is a conjugate of a subgroup in $\Pc$), fix an enumeration $P=\{g_1,g_2,\dots\}$, and let $Q_m:=\{g_1,\dots,g_m\}\subset P$ for each positive integer $m$. 

Fix a Groves--Manning cusp space $X$ for $(\Gamma,\Pc)$, and let $\Gamma^\prime$ be the set of elements $\gamma \in \Gamma$ for which there is a geodesic ray $\sigma : [0,\infty) \rightarrow X$ such that $\sigma(0)=\gamma$, $\sigma(T) = \id$ for some $T \ge 0$ and $\lim_{t \rightarrow \infty} \sigma(t) = \xi^{-1}(F_0)\in\partial_\infty X=\partial(\Gamma,\Pc)$. The next lemma implies that for all positive integers $m$,
\[V_m:=(\Gamma\cup\Lambda_\theta(\Gamma))- Q_m \cdot \overline{\Gamma^\prime}\]
is an open neighborhood of $F_0$, where $\overline{\Gamma'}$ denotes the closure of $\Gamma'$ in $\Gamma\cup\Lambda_\theta(\Gamma)$.

\begin{lemma} \label{lem: Vn construction}
The closed subset $\overline{\Gamma'}\subset \Gamma\cup\Lambda_\theta(\Gamma)$ does not contain $F_0$. In particular, for all positive integers $m$, $V_m\subset\Gamma\cup\Lambda_\theta(\Gamma)$ is an open neighborhood of $F_0$.
\end{lemma} 

\begin{proof} 
Suppose for a contradiction that there exists a sequence $\{\gamma_n\}$ in $\Gamma^\prime$ converging to $F_0$. Then  Lemma~\ref{lem: limit set compactifies} implies that $U_\theta(\gamma_n) \rightarrow F_0$. Hence, if we view $\{\gamma_n\}$ as a sequence in $X$, then by Proposition~\ref{prop:compactifications are the same} and taking a subsequence if necessary, we must have $\gamma_n=\gamma_n(\id) \rightarrow \xi^{-1}(F_0)$. 

For each $n$ fix a geodesic ray $\sigma_n : [0,\infty) \rightarrow X$ such that $\sigma_n(0)=\gamma_n$, $\sigma_n(T_n) = \id$ for some $T_n \ge 0$ and $\lim_{t \rightarrow \infty} \sigma_n(t) = \xi^{-1}(F_0)$. Passing to a subsequence we can suppose that the geodesics $\{\sigma_n(\cdot+T_n)\}$ converges locally uniformly to a geodesic $\sigma : \Rb \rightarrow X$. Then, since $\gamma_n \rightarrow \xi^{-1}(F_0)$, we have 
$$
\lim_{t \rightarrow - \infty} \sigma(t) =\lim_{n \rightarrow \infty} \sigma_n(0) =\lim_{n \rightarrow \infty} \gamma_n=  \xi^{-1}(F_0)=\lim_{t \rightarrow + \infty} \sigma(t),
$$
which is impossible. 

Since $Q_m$ is finite for all $m$, the second claim follows immediately from the first.
\end{proof} 

Since $P$ is conjugate to a subgroup in $\Pc$, by Corollary \ref{cor:entropy gap with peripherals}, 
$$\epsilon:=\frac{\delta^\phi(\Gamma)-\delta^\phi(P)}{2}>0.$$
To prove that Equation \eqref{eqn: no parabolic atoms} holds, we prove the following lemma. This will be used to give an upper bound on $\mu_s(V_m)$ involving the tail of the Poincar\'e series $Q_P^\phi(\delta^\phi(P)+\epsilon)$.

\begin{lemma}\label{lem:multiplicative with gamma prime} \ \begin{enumerate}
\item $P \cdot \Gamma^\prime = \Gamma$. 
\item There exists $C_1 > 0$ such that 
$$
\left| \phi(\kappa_\theta(g\gamma))-\phi(\kappa_\theta(g))-\phi(\kappa_\theta(\gamma)) \right| \le C_1
$$
for all $\gamma \in \Gamma^\prime$ and $g \in P$.
\item There exists $C_2 > 0$ such that 
$$
 h\left(e^{\phi(\kappa_\theta(g_k))+\phi(\kappa_\theta(\gamma))+C_1}\right) \le C_2 e^{\epsilon\phi(\kappa_\theta(g_k))}h\left(e^{\phi(\kappa_\theta(\gamma))}\right)
$$
for all $\gamma \in \Gamma^\prime$ and $k \ge 1$. 
\end{enumerate}

\end{lemma}

\begin{proof}
Proof of (1). Fix $\gamma \in \Gamma$. Then fix a geodesic ray $\sigma : [0,\infty) \rightarrow X$ with $\sigma(0) = \gamma$ and $\lim_{t \rightarrow \infty} \sigma(t) = \xi^{-1}(F_0)$. Let $\mathcal{H}_P\subset X$ denote the combinatorial horoball associated to $P$. Notice that the geodesic ray $\sigma_0 : [0,\infty) \rightarrow X$ which satisfies $\sigma_0(n) = (\id,n) \in \mathcal{H}_P$ for all $n \in \Nb$ also limits to $\xi^{-1}(F_0)$. Hence 
$$
\sup_{t \ge 0} \, \d_X( \sigma(t), \sigma_0(t)) < +\infty.
$$
Since 
$$
\lim_{t \rightarrow \infty} \d_X(\sigma_0(t), X - \mathcal{H}_P) = +\infty, 
$$
 there exists $T \ge 0$ such that  $\sigma(T) \in P$ and $\sigma(t) \in \mathcal{H}_P$ for all $t \ge T$. Then $\sigma(T)^{-1} \gamma \in \Gamma^\prime$. Since $\gamma\in\Gamma$ was arbitrary, (1) holds.

Proof of (2). Suppose not. Then for each $n \ge 1$ there exist $\gamma_n \in \Gamma^\prime$ and $g_n \in P$ such that 
$$
\abs{ \phi(\kappa_\theta(g_n\gamma_n))-\phi(\kappa_\theta(g_n))-\phi(\kappa_\theta(\gamma_n)) } \ge n.
$$
Lemma~\ref{lem:kappa multiplication estimate} implies that $\{\gamma_n\}$ and $\{g_n\}$ are both escaping sequences in $\Gamma$. Since $\{g_n\} \subset P$, in the topology on $\Gamma \cup \Lambda_\theta(\Gamma)$ we have $g_n^{-1} \rightarrow F_0$. Passing to a subsequence we can also assume that 
$$
F : = \lim_{n \rightarrow \infty} \gamma_n \in \Lambda_\theta(\Gamma)
$$
exists. By Lemma \ref{lem: Vn construction}, $F_0 \neq F$. So by Lemma~\ref{lem: limit set compactifies} and Proposition~\ref{prop:multiplicative estimate v1}, we have 
$$
\limsup_{n \rightarrow \infty} \abs{ \phi(\kappa_\theta(g_n\gamma_n))-\Big(\phi(\kappa_\theta(g_n))+\phi(\kappa_\theta(\gamma_n))\Big) } < +\infty
$$
and hence a contradiction. 

Proof of (3). Let $\nu_0>0$ be a constant such that if $s > 1$ and $\nu > \nu_0$, then $h(\nu s) \le s^{\epsilon} h(\nu)$. Let $C_1>0$ be the constant from (2), and fix $C_1^\prime \ge C_1$ such that 
$$
\phi(\kappa_\theta(g_k)) +C_1^\prime > 0 \quad \text{and} \quad \phi(\kappa_\theta(\gamma)) +C_1^\prime > \log \nu_0
$$
for all $\gamma \in \Gamma^\prime$ and $k \ge 1$. Also let 
$$
\Gamma^{\prime\prime} : = \{ \gamma \in \Gamma^\prime : \phi(\kappa_\theta(\gamma)) \le \log \nu_0\}.
$$
If $\gamma \notin \Gamma^{\prime\prime}$, then 
$$
 h\left(e^{\phi(\kappa_\theta(g_k))+\phi(\kappa_\theta(\gamma))+C_1}\right) \le  h\left(e^{\phi(\kappa_\theta(g_k))+\phi(\kappa_\theta(\gamma))+C_1^\prime }\right) \le e^{\epsilon C_1^\prime} e^{\epsilon\phi(\kappa_\theta(g_k))}h\left(e^{\phi(\kappa_\theta(\gamma))}\right).
$$
On the other hand, if $\gamma \in \Gamma^{\prime\prime}$, then
$$
 h\left(e^{\phi(\kappa_\theta(g_k))+\phi(\kappa_\theta(\gamma))+C_1}\right) \le  h\left(e^{\phi(\kappa_\theta(g_k))+\phi(\kappa_\theta(\gamma))+2C_1^\prime }\right) \le e^{\epsilon C_1^\prime} e^{\epsilon\phi(\kappa_\theta(g_k))}h\left(e^{\phi(\kappa_\theta(\gamma))+C_1^\prime}\right).
$$
So (3) holds with
\[
C_2 : = e^{\epsilon C_1^\prime} \max \left\{ \frac{h\left(e^{\phi(\kappa_\theta(\gamma))+C_1^\prime}\right)}{h\left(e^{\phi(\kappa_\theta(\gamma))}\right)} : \gamma \in \Gamma^{\prime\prime}\right\},
\]
which is finite since $\Gamma^{\prime\prime}$ is finite. 
\end{proof} 

If $s>\delta$, then by Lemma \ref{lem:multiplicative with gamma prime},
\begin{align*}
\mu_s(V_m) & \le \frac{1}{\hat{Q}(s)} \sum_{k > m} \sum_{\gamma \in \Gamma^\prime} h\left(e^{\phi(\kappa_\theta(g_k\gamma))}\right)e^{-s\phi(\kappa_\theta(g_k\gamma))}\\
& \le \frac{e^{sC_1}}{\hat{Q}(s)} \sum_{k > m} \sum_{\gamma \in \Gamma^\prime} h\left(e^{\phi(\kappa_\theta(g_k))+\phi(\kappa_\theta(\gamma))+C_1}\right)e^{-s\phi(\kappa_\theta(g_k)) -s\phi(\kappa_\theta(\gamma))}\\
& \le \frac{C_2e^{sC_1}}{\hat{Q}(s)} \sum_{k > m} \sum_{\gamma \in \Gamma^\prime} e^{-(s-\epsilon)\phi(\kappa_\theta(g_k))}h\left(e^{\phi(\kappa_\theta(\gamma))}\right)e^{-s\phi(\kappa_\theta(\gamma))}\\
& \le C_2e^{sC_1}\sum_{k > m} e^{-(\delta^\phi(P)+\epsilon)\phi(\kappa_\theta(g_k))}.
\end{align*}
Since $Q_{P}^\phi(\delta^\phi(P)+\epsilon) < +\infty$, it follows that
\[
\liminf_{m \rightarrow \infty}\,  \limsup_{s \searrow \delta} \, \mu_{s}(V_m) \le C_2 e^{\delta C_1} \liminf_{m \rightarrow \infty}\sum_{k >m} e^{-(\delta^\phi(P)+\epsilon)\phi(\kappa_\theta(g_k))}=0,
\]
so Equation \eqref{eqn: no parabolic atoms} holds.
\end{proof}

\begin{remark}[The elementary case]\label{rmk:main theorem in elementary case} In this remark we sketch why the Poincar\'e series of an  infinite elementary transverse groups diverges at its critical exponent. 

Suppose $\Gamma\subset \GG$ is an infinite $\Psf_\theta$-transverse subgroup, $\#\Lambda_\theta(\Gamma) \le 2$,  $\phi\in \mathfrak{a}_\theta^*$ and $\delta^\phi(\Gamma) < +\infty$. 

\medskip 
\noindent \emph{Case 1:} Suppose $\Lambda_\theta(\Gamma)  =\{ F^+, F^-\}$. Then one can show that there is an infinite order element $\gamma \in \Gamma$ such that $\langle\gamma\rangle$ has finite index in $\Gamma$. Moreover, we can assume that  $\gamma^n(F) \rightarrow F^+$ for all $F  \in \Fc_\theta$ transverse to $F^-$. Then arguing as in Case 2 of the proof of Corollary \ref{cor:entropy gap rel anosov 1} below, one can show that $\delta^\phi(\Gamma) = 0$. Hence, since $\Gamma$ is infinite, we have
$$
Q_\Gamma^\phi(0) = \sum_{\gamma \in \Gamma} 1 = +\infty. 
$$

\medskip 
\noindent \emph{Case 2:} Suppose $\Lambda_\theta(\Gamma)  =\{ F^+\}$. Using Proposition~\ref{prop:reducing to no simple factors} we may assume that $\GG$ has trivial center, and that $\Psf_\theta$ contains no simple factors of $\GG$. Then one can adapt the proof of Theorem~\ref{thm:properties of relatively Anosov representations} part (2) to show that there exists a closed subgroup $\mathsf{H} \subset \GG$ with finitely many components such that: \begin{enumerate}
\item $\Gamma$ is a cocompact lattice in $\mathsf{H}$.
\item $\mathsf{H} = \mathsf{L} \ltimes \mathsf{U}$ where $\mathsf{L}$ is compact and $\mathsf{U}$ is the unipotent radical of $\mathsf{H}$. 
\item $\mathsf{H}^0 = \mathsf{L}^0 \times \mathsf{U}$ and $\mathsf{L}^0$ is Abelian. 
\end{enumerate} 
Finally, one can use the proof of Theorem \ref{thm:entropy gap for rel Anosov} to show that $Q_\Gamma^\phi$ diverges at its critical exponent.

\end{remark}

\section{Relatively quasiconvex subgroups}\label{sec: quasiconvex}

 Suppose that $\Gamma$ is hyperbolic relative to $\Pc$. Given a subgroup $\Gamma_0 \subset \Gamma$, the \emph{limit set of $\Gamma_0$}, denoted $\Lambda(\Gamma_0)$,  
is the set of all points $x \in \partial(\Gamma, \Pc)$ where there is a sequence $\{h_n\}$ in $\Gamma_0$ and $y \in \partial(\Gamma, \Pc)$
such that $h_n(z) \rightarrow x$ for all $z \in \partial(\Gamma, \Pc) -\{y\}$. Notice that the action of $\Gamma_0$ on $\Lambda(\Gamma_0)$ is a convergence group action. Then $\Gamma_0$ is \emph{relatively quasiconvex} if $\Gamma_0$ acts geometrically finitely on $\Lambda(\Gamma_0)$, see~\cite[Defn.\ 6.2]{Hruska}. In this section we prove that infinite index, relatively quasiconvex subgroups of a relatively Anosov group
 have strictly lower critical exponent. 
 
\begin{corollary}[Corollary \ref{cor:entropy gap rel anosov}]
\label{cor:entropy gap rel anosov 1}
Suppose $\Gamma \subset \GG$ is a $\Psf_\theta$-Anosov subgroup relative to $\Pc$, $\phi \in \mfa_\theta^*$ and $\delta^\phi(\Gamma) < +\infty$. 
If $\Gamma_0$ is infinite index relatively quasiconvex subgroup of $(\Gamma, \Pc)$, then
$$\delta^\phi(\Gamma)>\delta^\phi(\Gamma_0).$$
\end{corollary}

\begin{remark} \label{rem: Hruska}
Since we will be using results from Hruska \cite{Hruska} extensively in this section, we remark that the definition of relatively hyperbolic group used in \cite[Defn.\ 3.1]{Hruska} is strictly weaker than our Definition~\ref{defn:RH}. To obtain our Definition~\ref{defn:RH} from Hruska's definition, one needs to further require that $\Gamma$ 
be non-elementary and finitely generated, and that every $P\in\Pc$ is finitely generated. By \cite[Cor.\ 9.2]{Hruska}, these additional conditions follow if $\Gamma$ is non-elementary and every $P\in\Pc$ is \emph{slender}, i.e. any subgroup of $P$ is finitely generated. \end{remark}

As mentioned in the Introduction, the two results needed to deduce Corollary~\ref{cor:entropy gap rel anosov 1} from Theorems \ref{intro: rel anosov divergent} and \ref{thm: entropy gap CZZ3} are stated as Lemma \ref{lem:limit set is a proper subset} and Proposition \ref{prop:rel quasiconvex are rel Anosov} below. They both follow from results of Hruska \cite{Hruska}.

\begin{lemma}\label{lem:limit set is a proper subset} Suppose $\Gamma$ is hyperbolic relative to $\Pc$ and $\Gamma_0 \subset \Gamma$ is relatively quasiconvex. If $\Gamma_0$ has infinite index in $\Gamma$, then $\Lambda(\Gamma_0)$ is a closed proper subset of $\partial(\Gamma, \Pc)$. \end{lemma}

\begin{proof} We will prove the contrapositive: if $\Lambda(\Gamma_0) = \partial(\Gamma, \Pc)$, then $\Gamma_0$ has finite index in $\Gamma$. 

Let $X$ be  a Groves--Manning cusp space for $(\Gamma, \Pc)$. Given a subset $A \subset X$, we will let $\mathcal{N}_r(A)$ denote the closed $r$-neighborhood of $A$ in $X$. Since $X$ is Gromov hyperbolic, there exists $c> 0$ such that if $\sigma_1, \sigma_2 : \Rb \rightarrow X$ are geodesics with 
$$
\lim_{t \rightarrow \pm \infty} \sigma_1(t) = \lim_{t \rightarrow \pm \infty} \sigma_2(t),
$$
then 
$$
\sigma_1 \subset \mathcal{N}_c(\sigma_2) \quad \text{and} \quad \sigma_2 \subset \mathcal{N}_c(\sigma_1).
$$

Fix a geodesic $\sigma_0 : \Rb \rightarrow X$. Then fix $r > 0$ such that 
$$
\id \in \mathcal{N}_r(\sigma_0). 
$$
 An equivalent definition \cite[Defn.\ 6.6]{Hruska} of relatively quasiconvex subgroups implies that there exists $R > 0$ such that: if $s : [0,T] \rightarrow X$ is a geodesic segment with endpoints in $\Gamma_0$, then 
$$
s \cap \mathcal{N}_{c+r+1}(\Gamma) \subset \mathcal{N}_R(\Gamma_0),
$$
see \cite[Prop.\ 7.5 and 7.6]{Hruska}.

Fix $g \in \Gamma$ and let $\sigma : = g \circ \sigma_0 : \Rb \rightarrow X$. Since $\Lambda(\Gamma_0) = \partial(\Gamma, \Pc)$, there exist sequences $\{h_n^-\}$, $\{h_n^+\}$ in $\Gamma_0$ such that 
$$
\lim_{n \rightarrow \infty} h_n^\pm = \lim_{t \rightarrow \pm \infty} \sigma(t). 
$$
Let $\sigma_n$ be a geodesic in $X$ joining $h_n^-$ to $h_n^+$. After possibly reparametrizing and passing to a subsequence we may suppose that $\sigma_n$ converges locally uniformly to a geodesic  $\sigma_\infty : \Rb \rightarrow X$. Then
$$
\sigma_\infty \cap \mathcal{N}_{c+r}(\Gamma) \subset \mathcal{N}_{R}(\Gamma_0).
$$
Further,
$$
\lim_{t \rightarrow \pm \infty} \sigma_\infty(t) = \lim_{n \rightarrow \infty} h_n^\pm = \lim_{t \rightarrow \pm \infty} \sigma(t),
$$
and so $g \in \mathcal{N}_{r}(\sigma) \subset \mathcal{N}_{c+r}(\sigma_\infty)$. Hence,
$$
g \in \mathcal{N}_{R+c+r}(\Gamma_0). 
$$
Since $g \in \Gamma$ was arbitrary, 
$$
\Gamma \subset \mathcal{N}_{R+c+r}(\Gamma_0) 
$$
and thus $\Gamma_0 \subset \Gamma$ has finite index. 
\end{proof}

\begin{proposition}\label{prop:rel quasiconvex are rel Anosov} Suppose $\Gamma \subset \GG$ is a $\Psf_\theta$-Anosov subgroup relative to $\Pc$. Assume $\Gamma_0 \subset \Gamma$ is 
non-elementary and relatively quasiconvex. Let $\Pc_0$ denote a set of representatives of the conjugacy classes in $\Gamma_0$ of the intersection of the peripheral subgroups of $\Gamma$ with $\Gamma_0$. Then: 
\begin{enumerate}
\item $(\Gamma_0, \Pc_0)$ is relatively hyperbolic (in the sense of Definition~\ref{defn:RH}).
\item $\Gamma_0$ is a $\Psf_\theta$-Anosov subgroup relative to $\Pc_0$. 
\end{enumerate} 
\end{proposition} 

\begin{proof}
Proof of (1). Hruska \cite[Thm.\ 9.1]{Hruska} proved that relatively quasiconvex subgroups of relatively hyperbolic groups in his weaker sense (see Remark \ref{rem: Hruska}) are also relatively hyperbolic. Thus, it suffices to show that every $P\in\Pc_0$ is slender. 

\begin{lemma}\label{lem:peripheral subgroups are slender} Suppose $\Gamma \subset \GG$ is a $\Psf_\theta$-Anosov subgroup relative to $\Pc$. If $P \in \Pc$ and $Q \subset P$ is a subgroup, then $Q$ is finitely generated. 
\end{lemma} 

\begin{proof} Using Proposition~\ref{prop:reducing to no simple factors} we may assume that $Z(\GG)$ is trivial and $\Psf_\theta$ contains no simple factors of $\GG$ (notice that if $p : \GG \rightarrow \GG^\prime$ is as in the proposition, then $p|_\Gamma$ has finite kernel and hence $Q$ is finitely generated if and only if $p(Q)$ is finitely generated). By Theorem~\ref{thm:properties of relatively Anosov representations} there exists a closed subgroup $\mathsf{H} \subset \GG$ with finitely many components such that: 
\begin{enumerate}
\item $P$ is a cocompact lattice in $\mathsf{H}$.
\item $\mathsf{H} = \mathsf{L} \ltimes \mathsf{U}$ where $\mathsf{L}$ is compact and $\mathsf{U}$ is the unipotent radical of $\mathsf{H}$. 
\end{enumerate} 
Then Auslander's theorem (see for instance~\cite[Thm.\ 11.1]{KL}) implies that $Q$ is finitely generated.

\end{proof}


Proof of (2). By definition, $\Gamma_0$ is a $\Psf_\theta$-transverse subgroup of $\GG$. Also, by (1), $(\Gamma_0, \Pc_0)$ is relatively hyperbolic, and so we may identify 
$$
\partial(\Gamma_0,\Pc_0)=\Lambda(\Gamma_0) \subset \partial(\Gamma, \Pc). 
$$
Since $\Gamma$ is a $\Psf_\theta$-Anosov subgroup relative to $\Pc$, there is a $\Gamma$-equivariant homeomorphism $\xi: \partial(\Gamma, \Pc) \rightarrow \Lambda_\theta(\Gamma)$. Thus, to show that $\Gamma_0$ is a $\Psf_\theta$-Anosov subgroup relative to $\Pc_0$, it suffices to show that $\xi(  \Lambda(\Gamma_0)) = \Lambda_\theta(\Gamma_0)$. 

Fix any $F^+ \in \Lambda_\theta(\Gamma_0)$. Then by Proposition~\ref{prop:characterizing convergence in general symmetric case} there is $\{h_n\}$ in $\Gamma_0$ and $F^- \in \Lambda_\theta(\Gamma)$ such that $h_n(F) \rightarrow F^+$ for all $F \in \Fc_\theta$ transverse to $F^-$. By definition, $F^\pm = \xi(x^\pm)$ for some $x^\pm \in \partial(\Gamma, \Pc)$. Also, since $\Gamma_0$ acts on $\Lambda(\Gamma_0)$ as a convergence group, by passing to a subsequence, we can suppose that there exist $y^\pm \in \Lambda(\Gamma_0)$ such that $h_n(z) \rightarrow y^+$ for all $z \in \Lambda(\Gamma_0) -\{y^-\}$. Since $\Gamma_0$ is non-elementary, we may fix $z \in  \Lambda(\Gamma_0) -\{x^-, y^-\}$. Then 
$$
F^+ = \lim_{n \rightarrow \infty} h_n (\xi(z)) = \lim_{n \rightarrow \infty} \xi(h_n(z)) = \xi( y^+). 
$$
Since $F^+$ was arbitrary, it follows that $\Lambda_\theta(\Gamma_0) \subset \xi(  \Lambda(\Gamma_0))$. A very similar argument shows that $\xi(  \Lambda(\Gamma_0)) \subset \Lambda_\theta(\Gamma_0)$. 
\end{proof} 


We may now give the proof of Corollary~\ref{cor:entropy gap rel anosov 1}.

\begin{proof}[Proof of Corollary~\ref{cor:entropy gap rel anosov 1}] First, notice that if $\Gamma_0\subset\Gamma$ is an elementary subgroup, then it is either finite (in which case $\Lambda (\Gamma_0)$ is empty), conjugate to a subgroup of a peripheral subgroup of $\Gamma$
(in which case $\Lambda(\Gamma_0)$ is a single point)
or virtually a cyclic group generated by a hyperbolic element (in which case $\Lambda(\Gamma_0)$ consists of the the attracting and repelling fixed points of the hyperbolic element). If $\Gamma_0$ is non-elementary,
then $\Lambda(\Gamma_0)$ is perfect. See the discussion in \cite[Sec.\ 3.1]{Hruska} for more details.

\medskip 
\noindent \emph{Case 1:}  Suppose $\Gamma_0$ is non-elementary. Let $\Pc_0$ denote a set of representatives of the conjugacy classes in $\Gamma_0$ of the intersection of the peripheral subgroups of $\Gamma$ with $\Gamma_0$. Then Proposition~\ref{prop:rel quasiconvex are rel Anosov} implies that $\Gamma_0$ is a $\Psf_\theta$-Anosov subgroup relative to $\Pc_0$. By Lemma~\ref{lem:limit set is a proper subset}, $\Lambda_\theta(\Gamma_0)$ is a proper subset of $\Lambda_\theta(\Gamma)$. By Theorem~\ref{intro: rel anosov divergent}, $Q^\phi_{\Gamma_0}$ converges at its critical exponent. So, by Theorem~\ref{thm: entropy gap CZZ3}, 
$$
\delta^\phi(\Gamma_0) < \delta^\phi(\Gamma). 
$$

\medskip 
\noindent \emph{Case 2:} Assume $\#\Lambda(\Gamma_0) = 2$. Then there is an infinite order hyperbolic element $\gamma \in \Gamma_0$ such that $\langle\gamma\rangle$ has finite index in $\Gamma_0$. Since $\Gamma$ acts as a convergence group on $\Lambda_\theta(\Gamma)$, we can label the fixed 
points $F^+, F^- \in \Lambda_\theta(\Gamma)$ of $\gamma$ so that $\gamma^n(F) \rightarrow F^+$ for all $F  \in \Fc_\theta$ transverse to $F^-$. 
By Proposition~\ref{prop:characterizing convergence in general symmetric case}, $U_\theta(\gamma^{-n}) \rightarrow F^-$ and $U_\theta(\gamma^n) \rightarrow F^+$. Then by Proposition~\ref{prop:multiplicative estimate v1} there exists $C > 0$ such that 
$$
\phi(\kappa_\theta(\gamma^{n+m})) \ge \phi(\kappa_\theta(\gamma^n)) + \phi(\kappa_\theta(\gamma^m)) -C
$$
for all $n,m \ge 1$. This estimate implies that $\delta^\phi(\langle\gamma\rangle) = 0$, and hence that $\delta^\phi(\Gamma_0)=0$.

Since $\Gamma$ is non-elementary,  it contains a free subgroup of rank two and hence 
$$
\delta^\phi(\Gamma_0)=0< \delta^\phi(\Gamma).
$$

\medskip 
\noindent \emph{Case 3:} Assume $\#\Lambda(\Gamma_0) = 1$. Then, after conjugating, there is a peripheral subgroup $P \in \Pc$ with $\Gamma_0 \subset P$. Then by Corollary~\ref{cor:entropy gap with peripherals} we have 
$$
\delta^\phi(\Gamma_0) \le \delta^\phi(P) < \delta^\phi(\Gamma). 
$$

\medskip 
\noindent \emph{Case 4:} Assume $\#\Lambda(\Gamma_0) = 0$. Then $\Gamma_0$ is finite and so $\delta^\phi(\Gamma_0) = 0$. 
So, as in Case 2,
\[
\delta^\phi(\Gamma_0) = 0< \delta^\phi(\Gamma).\qedhere
\]

\end{proof}

\section{Characterizing linear functions with finite critical exponent}
\label{sec:characterizing finite entropy functionals}

In this section, we give a complete analysis of which linear functionals in $\mathfrak a_\theta^*$ have associated Poincar\'e series with finite critical exponents. This generalizes the results of Sambarino~\cite{sambarino-dichotomy} for Anosov groups.

Given a subgroup $\Gamma\subset\GG$, the \emph{$\theta$-Benoist limit cone of $\Gamma$}, denoted $B_\theta(\Gamma)\subset \mathfrak{a}_\theta$, is the set of vectors $X \in \mathfrak{a}_\theta$ for which there exists a sequence $\{\gamma_n\}$ of distinct elements of $\Gamma$ and a 
sequence $\{r_n\}$ in $\mathbb R_+$ so that $r_n\kappa_\theta(\gamma_n)$ converges to $X$. 

As in Proposition~\ref{prop:reducing to no simple factors}, let  $p : \GG \rightarrow \GG^\prime$ denote the projection map of $\GG$ onto $\GG^\prime:=\GG/\mathsf{H}$, where $\mathsf{H}$ is the product of $Z(\GG)$ and the simple factors of $\GG$ contained in $\Psf_\theta$. Notice that the Benoist limit cones of a group $\Gamma \subset \GG$ and its projection $\Gamma^\prime: = p(\Gamma) \subset \GG^\prime$ satisfy
$$
\d p(B_\theta(\Gamma)) = B_{\theta^\prime}(\Gamma^\prime).
$$

We prove the following expanded version of Theorem~\ref{thm:characterizing finite entropy functionals}.

\begin{theorem}\label{thm:characterizing finite entropy functionals in paper} Suppose $\Gamma\subset \GG$ is a $\Psf_\theta$-Anosov subgroup relative to $\Pc$ and  $\phi\in \mathfrak{a}_\theta^*$. The following are equivalent: 
\begin{enumerate}
\item $\lim_{n \rightarrow \infty} \phi(\kappa_\theta(\gamma_n))=+\infty$ for every sequence of distinct elements $\{\gamma_n\}$ in $\Gamma$.
\item $\delta^\phi(\Gamma) < +\infty$. 
\item 
If  $x_0\in M' : =\GG'/\Ksf'$, where $\mathsf{K}':=p(\mathsf{K})$, then there exist constants $c\ge 1,C \ge 0$ such that 
$$
\frac{1}{c} \d_{M'}(\gamma(x_0), x_0)- C \le \phi(\kappa_\theta(\gamma)) \le c \d_{M'}(\gamma(x_0), x_0)+ C
$$
for all $\gamma \in \Gamma$ (where $\d_{M'}$ is the distance defined on $M'$ in Section~\ref{sec: ss Lie group background}). 
\item $\phi(Y)>0$ for all $Y\in B_{\theta}(\Gamma)-\{0\}$.
\item If $X$ is a Groves--Manning cusp space for $(\Gamma, \Pc)$, then there exist constants $c,C > 0$ such that 
$$
 \phi(\kappa_\theta(\gamma)) \ge c \d_X(\gamma, \id) - C
$$
for all $\gamma \in \Gamma$. 

\end{enumerate} 

\end{theorem} 

Indeed, statements (1) and (2) of Theorem~\ref{thm:characterizing finite entropy functionals} are the same as statements (1) and (2) of Theorem~\ref{thm:characterizing finite entropy functionals in paper}. Furthermore, when $\Psf_\theta$ contains no simple factors of $\GG$, then $\GG^\prime = \GG / Z(\GG)$. Thus, in this case $M = M^\prime$ and so statement (3) of Theorem~\ref{thm:characterizing finite entropy functionals} is equivalent to statement (3) of  Theorem~\ref{thm:characterizing finite entropy functionals in paper}.

We will now prove Theorem~\ref{thm:characterizing finite entropy functionals in paper}. By part (3) of Proposition~\ref{prop:reducing to no simple factors}, we may assume that $\GG$ has trivial center and $\Psf_\theta$ contains no simple factors of $\GG$, in which case $M'=M:=\GG/\Ksf$. Then note that any one of (2), (3), (4) or (5) immediately imply (1), and (3) also immediately implies (4). Since $\Psf_\theta$ contains no simple factors of $\GG$, the equivalence of (3) and (5) follows from Theorem~\ref{thm:properties of relatively Anosov representations}. So it suffices to show that (3) implies (2) and (1) implies (5).

\medskip\noindent
{\em Proof of} (3) $\Rightarrow$ (2): Without loss of generality we may assume that $x_0 := \Ksf \in M=\GG/\Ksf$. By assumption, there exist constants $c,C > 0$ such that 
$$
 \phi(\kappa_\theta(\gamma)) \ge c \d_M(\gamma(x_0), x_0)- C
$$
for all $\gamma \in \Gamma$. Then $\delta^\phi(\Gamma)\le \frac{1}{c}\delta_M(\Gamma)$, where 
\[
\delta_M(\Gamma):=\lim_{T\to\infty}\frac{\log\#\left\{\gamma\in\Gamma:\d_M( \gamma(x_0), x_0) <T\right\}}{T}.
\]

Recall that the volume growth entropy of $M$ is 
$$
h(M) : = \limsup_{T \rightarrow \infty} \frac{\log {\rm Vol}_M( B_T(x_0) )}{T} 
$$
where ${\rm Vol}_M$ is the Riemannian volume on $M$ and $B_r(x_0) \subset M$ is the open ball of radius $r > 0$. Since $M$ has bounded sectional curvature, volume comparison theorems imply that $h(M) <+\infty$.

Fix $r_0 > 0$ and for $T > 0$ let $\Gamma_T : = \left\{\gamma\in\Gamma:\d_M(\gamma(x_0), x_0) <T\right\}$. Then
$$
\#\Gamma_T = \frac{1}{{\rm Vol}_M(B_{r_0}(x_0))} \sum_{\gamma \in \Gamma_T} {\rm Vol}_M(B_{r_0}(\gamma x_0)) \le \frac{\#\Gamma_{2r_0} }{{\rm Vol}_M(B_{r_0}(x_0))} {\rm Vol}_M(B_{T+r_0}(x_0)). 
$$
Thus $\delta_M(\Gamma) \le h(M)< +\infty$.
\qed

\medskip

The proof that (1) implies (5) is more technical, so we provide  a brief outline. We first use Proposition~\ref{prop:properties of unipotent subgroups} to provide a lower bound for $\phi\circ\kappa_\theta$
on peripheral subgroups, see Lemma~\ref{lem: good growth along peripherals}. We then divide a geodesic joining $\id$ to $\gamma\in\Gamma$ in the Groves--Manning cusp space $X$ into segments $\overline{\gamma_i\gamma_{i+1}}$
with endpoints in $\Gamma$ which either (a) have a pre-chosen size guaranteeing that   $\phi(\kappa_\theta(\rho(\gamma_{i+1}\gamma_i^{-1})))$ is large enough, or (b) are at least
as long as the pre-chosen size and lie entirely in a cusped portion of $X$. We then apply Proposition~\ref{prop:multiplicative estimate v2} to show that the image
of the Cartan projections are roughly additive along the segment.

\medskip\noindent
{\em Proof of} (1) $\Rightarrow$ (5): Suppose that 
$$
\lim_{n \rightarrow \infty} \phi(\kappa_\theta(\gamma_n))=+\infty
$$
for every sequence of distinct elements $\{\gamma_n\}$ in $\Gamma$.

Fix a Groves--Manning cusp space $X$ for $(\Gamma, \Pc)$. We first control the growth of peripheral elements.

\begin{lemma}\label{lem: good growth along peripherals} There exist $c_1, C_1 > 0$ such that: if $P \in \Pc$ and $\upsilon \in P$, then 
$$
\phi(\kappa_\theta(\upsilon)) \ge c_1 \d_X(\upsilon, \id) - C_1.
$$
\end{lemma} 

\begin{proof}
Recall that $\d_M$ was defined so that $\d_M(g\Ksf, \Ksf) = \norm{\kappa(g)}$ for all $g \in \GG$. Then by  Theorem~\ref{thm:properties of relatively Anosov representations} it suffices to find $c_1, C_1 > 0$ such that: if $P \in \Pc$ and $\upsilon \in P$, then 
$$
\phi(\kappa_\theta(\upsilon)) \ge c_1 \norm{\kappa(\upsilon)} - C_1. 
$$
Then, since $\Pc$ is finite, it is enough to fix $P \in \Pc$ and find constants $c_P, C_P > 0$ such that: if $\upsilon \in P$, then 
$$
\phi(\kappa_\theta(\upsilon)) \ge c_P \norm{\kappa(\upsilon)} - C_P.
$$

By Theorem~\ref{thm:properties of relatively Anosov representations} there exists a closed subgroup $\mathsf{H} \subset \GG$ with finitely many components such that: \begin{enumerate}
\item $P$ is a cocompact lattice in $\mathsf{H}$.
\item $\mathsf{H} = \mathsf{L} \ltimes \mathsf{U}$ where $\mathsf{L}$ is compact and $\mathsf{U}$ is the unipotent radical of $\mathsf{H}$. 
\end{enumerate} 
Let $\mathfrak{u}$ denote the Lie algebra of $\mathsf{U}$. By Proposition~\ref{prop:properties of unipotent subgroups}, $\mathsf{U} = \exp(\mathfrak{u})$. By the same proposition, for any $\alpha \in \theta$ there exist $M_\alpha \in \Nb$, $C_\alpha>1$ and a positive everywhere defined rational function $R_\alpha: \mathfrak{u} \rightarrow \Rb$ where 
$$
\frac{1}{C_\alpha} R_\alpha(Y)^{1/M_\alpha} \le e^{\omega_\alpha(\kappa(\exp(Y)))} \le C_\alpha R_\alpha(Y)^{1/M_\alpha}
$$
for all $Y \in \mathfrak{u}$. 

Write $\phi = \sum_{\alpha \in \theta} c_\alpha \omega_\alpha$, and define $R : = \prod_{\alpha \in \theta} R_\alpha^{\abs{c_\alpha}/M_\alpha}$ and $C_\phi: = \prod_{\alpha \in \theta} C_\alpha^{\abs{c_\alpha}}$. Then 
\[\frac{1}{C_\phi} R(Y) \le e^{\phi(\kappa_\theta(\exp(Y)))} \le C_\phi R(Y)\] 
for all $s\in\Rb$. We proved in Lemma~\ref{lem: R is proper} that $R$ is positive and proper, so by Theorem~\ref{thm:critical exponent of products of rational functions}  there exist $c_2, \epsilon > 0$ such that 
$$
R(Y) \ge c_2(1+\norm{Y})^\epsilon
$$
for all $Y \in \mathfrak{u}$. By Proposition~\ref{prop:properties of unipotent subgroups}, there exist $A > 0$ such that 
$$
\norm{\kappa(\exp(Y)) } \le A+A\log (1+\norm{Y})
$$
for all $Y \in \mathfrak{u}$. Finally, let 
$$
M : = \max \{ \norm{\kappa(\ell)} : \ell \in \mathsf{L}\}.
$$

If $\upsilon \in P$, then $\upsilon = \ell \exp(Y)$ for some $\ell \in \mathsf{L}$ and $Y \in \mathfrak{u}$. So by Lemma~\ref{lem:kappa multiplication estimate},
\begin{align*}
\phi(\kappa_\theta(\upsilon)) & \ge \phi(\kappa_\theta(\exp(Y))) - M\norm{\phi}\\
& \ge \log R(Y) - M\norm{\phi}  - \log C_\phi \\
& \ge \epsilon \log(1+ \norm{Y}) - M\norm{\phi}  - \log C_\phi + \log c_2 \\
& \ge \frac{\epsilon}{A} \norm{\kappa(\exp(Y)) }- \frac{\epsilon}{A} - M\norm{\phi}  - \log C_\phi + \log c_2 \\
& \ge \frac{\epsilon}{A} \norm{\kappa(\upsilon) }- \frac{\epsilon}{A}M - \frac{\epsilon}{A} - M\norm{\phi}  - \log C_\phi + \log c_2. \qedhere
\end{align*}
\end{proof}

By Proposition~\ref{prop:multiplicative estimate v2} there exists $C_0 > 0$ such that: if $f : [0,T] \rightarrow X$ is a geodesic with $f(0)=\id$, and $f(t_1) ,f(t_2) \in \Gamma$ for some $0\le t_1 \le t_2\le T$, then 
\begin{equation}\label{eqn: multiplicative along geodesics} 
\abs{\phi\big(\kappa_\theta(f(t_2))\big) - \phi\big( \kappa_\theta(f(t_1))\big)- \phi\big( \kappa_\theta(f(t_1)^{-1}f(t_2)) \big)} \le C_0. 
\end{equation} 
By hypothesis, $\lim_{n \rightarrow \infty} \phi(\kappa_\theta(\gamma_n))=+\infty$  for every sequence of distinct elements $\{\gamma_n\}$ in $\Gamma$, so there exists $T_0 > 0$ such that: if $\gamma \in \Gamma$ and $\d_X(\gamma, \id) \ge T_0$, then 
\begin{equation}\label{eqn:thick translation} 
\phi(\kappa_\theta(\gamma)) > 1 + C_0. 
\end{equation}
Then let 
$$
C_2: = \max\{ \abs{\phi(\kappa_\theta(\gamma))} : \d_X(\gamma, \id) < T_0\},
$$
let 
$$
B : = \frac{2}{c_1 T_0}\left( 3C_0 + C_1+2C_2 + c_1 T_0\right) > 2,
$$
and let 
$$
c : = \min\left\{ \frac{c_1}{2}, \frac{1}{(B+2) T_0}\right\}\quad\text{and}\quad C:=C_2+cT_0. 
$$
We will show that (5) holds with $c$ and $C$ as described above.

Fix $\gamma \in \Gamma$ and let $f : [0,T] \rightarrow X$ be a geodesic with $f(0) = \id$ and $f(T) = \gamma$. 
If $T < T_0$, then 
$$
\phi(\kappa_\theta(\gamma)) \ge -C_2 \ge c  \d_X(\gamma, \id) - C. 
$$
If $T\ge T_0$, fix a partition 
$$
0 =t_0 < t_1 < \dots < t_m = T
$$
with the following properties: 
\begin{enumerate}
\item $t_{n+1} - t_n \ge T_0$ for $n =0, \dots, m-1$,
\item $\gamma_n:=f(t_n) \in \Gamma$ for $n=0,\dots, m$ and
\item if $0 =s_0 < s_1 < \dots < s_{m^\prime} = T$ is another partition with the first two properties, then $m^\prime \le m$. 
\end{enumerate} 
Then by Equation~\eqref{eqn: multiplicative along geodesics}, 
$$
\phi(\kappa_\theta(\gamma)) =\sum_{n=0}^{m-1} \Big(\phi(\kappa_\theta(\gamma_{n+1}))-\phi(\kappa_\theta(\gamma_n))\Big)\ge \sum_{n=0}^{m-1} \Big(\phi(\kappa_\theta(\gamma_n^{-1}\gamma_{n+1}))-C_0 \Big). 
$$
Thus to complete the proof it suffices to verify that 
\begin{align}\label{eqn: pieces}
\phi(\kappa_\theta(\gamma_n^{-1}\gamma_{n+1})) \ge c (t_{n+1}-t_n)+C_0
\end{align}
for each $n=0,\dots, m-1$. Indeed, if this were the case, then
\[\phi(\kappa_\theta(\gamma)) \ge c \sum_{n=0}^{m-1} (t_{n+1}-t_n)=cT \ge c\d_X(\gamma,\id)-C. \]

We will now prove Equation \eqref{eqn: pieces}. Fix $n \in \{0,\dots, m-1\}$. If $t_{n+1}-t_n < (B+2) T_0$, then Equation~\eqref{eqn:thick translation}  implies that
$$
\phi(\kappa_\theta(\gamma_n^{-1}\gamma_{n+1})) \ge 1+C_0 \ge c(t_{n+1}-t_n) + C_0. 
$$
If $t_{n+1} - t_n >  (B+2) T_0$, then by the maximality of the partition and the fact that $B+2>4$, there exist $P \in \Pc$, $\eta \in \Gamma$ and $a,b \in [t_n,t_{n+1}]$ such that: 
\begin{enumerate}
\item $a \in [t_n, t_n+T_0)$, $b \in (t_{n+1}- T_0, t_{n+1}]$ (hence $a<b$),
\item $f(a), f(b) \in \eta P$, and 
\item $f|_{[a,b]}$ is contained in the combinatorial horoball associated to $\eta P$. 
\end{enumerate} 
Then by applying Equation~\eqref{eqn: multiplicative along geodesics} to the geodesic $t\mapsto\gamma_n^{-1}f(t)$, 
\begin{align*}
\phi(\kappa_\theta(\gamma_n^{-1}\gamma_{n+1})) & \ge \phi(\kappa_\theta(\gamma_n^{-1}f(a)))+\phi(\kappa_\theta(f(a)^{-1}f(b)))+\phi(\kappa_\theta(f(b)^{-1}\gamma_{n+1}))-2C_0 \\
& \ge \phi(\kappa_\theta(f(a)^{-1}f(b)))  -2C_0- 2C_2,
\end{align*}
where the last equality holds by the definition of $C_2$. Since $f(a)^{-1}f(b) \in P$ and 
$$
 \d_X(f(a)^{-1}f(b), \id)=\d_X(f(a), f(b)) > B T_0,
 $$
by Lemma~\ref{lem: good growth along peripherals}, we then have 
 \begin{align*}
\phi(\kappa_\theta(\gamma_n^{-1}\gamma_{n+1}))& \ge c_1 \d_X(f(a)^{-1}f(b), \id) -2C_0-C_1-2C_2\\
& \ge \frac{c_1}{2} \d_X(f(a), f(b)) +C_0 + c_1 T_0 \\
& \ge c(t_n - t_{n-1}) + C_0. 
\end{align*}
This completes the proof.
\qed

\appendix

\section{Proof of Proposition \ref{prop:properties of unipotent subgroups}}\label{app: unipotent subgroups}

In this appendix, we prove Proposition \ref{prop:properties of unipotent subgroups}. We start with an observation about the linear case. 

\begin{lemma}\label{lem:estimate on sigma1 when linear} If $\mathsf{U} \subset \SL(d,\Rb)$ is a connected unipotent group with Lie algebra $\mathfrak{u}$, then there exist $C_0 > 1$ and a positive polynomial $P : \mathfrak{u} \rightarrow \Rb$ 
such that 
$$
\frac{1}{C_0} P(Y)^{1/2} \le \sigma_1(e^Y) \le C_0P(Y)^{1/2}
$$
for all $Y \in \mathfrak{u}$. 
\end{lemma}

\begin{proof} Define $P : \mathfrak{u} \rightarrow \Rb$ by 
$$
P(Y) = \sum_{1 \le i,j \le d} \left[e^Y\right]_{i,j}^2.
$$
Since $\mathfrak{u}$ is nilpotent, see ~\cite[Section 4.8]{BorelBook},  $P$ is a polynomial. Observe that the Euclidean norm $\norm{\cdot}_{\rm Euc}:\mathrm{End}(\Rb^d)\to\Rb$ and the first singular value $\sigma_1:\mathrm{End}(\Rb^d)\to\Rb$ are both norms on the vector space $\mathrm{End}(\Rb^d)$, so there exists $C_0>1$ such that
\[\frac{1}{C_0}\le\frac{\sigma_1(X)}{\norm{X}_{\rm Euc}}\le C_0\]
for all $X\in\mathrm{End}(\Rb^d)$. Since $P(Y)^{1/2}=\norm{e^Y}_{\rm Euc}$ for all $Y \in \mathfrak{u}$, the lemma follows.
\end{proof} 

\begin{proof}[Proof of Proposition~\ref{prop:properties of unipotent subgroups}] \
Proof of (1). Since $Z(\GG)$ is trivial, $\Usf$ is isomorphic to ${\rm Ad}(\Usf)$. Since each element of ${\rm Ad}(\Usf)$ is unipotent, the matrix logarithm 
$$
\log(A) = \sum_{n=1}^\infty (-1)^{n+1} \frac{(A-\id)^n}{n}
$$
is well defined on ${\rm Ad}(\Usf)$. So the exponential map of $\mathsf{SL}(\mathfrak{g})$ induces a diffeomorphism
\hbox{${\rm ad}(\mathfrak{u}) \rightarrow {\rm Ad}(\mathsf{U})$}, which implies that the exponential map of $\GG$ induces a diffeomorphism $\mathfrak{u} \rightarrow \mathsf{U}$.

Proof of (2). See for instance~\cite[Prop.\ 10.14]{LieGroupsSanMartin}. 

Proof of (3). See for instance \cite[Prop.\ 3.4.2]{ZimmerBook}. 

Proof of (4). Let $d := \dim \mathfrak{g}$ and fix a linear  identification $\mathfrak{g} = \Rb^d$. This induces identifications $\mathsf{GL}(\mathfrak{g})=\mathsf{GL}(d, \Rb)$ and $\mathfrak{sl}(\mathfrak{g})=\mathfrak{sl}(d,\Rb)$. Using the root space decomposition, we can pick our identification so that ${\rm ad}(\mathfrak{a})$ is a subgroup of the diagonal matrices in $\mathsf{GL}(\mathfrak{g})=\mathsf{GL}(d, \Rb)$

Since $Z(\mathsf{G})$ is trivial, ${\rm ad}: \mathfrak{g} \rightarrow \mathfrak{sl}(\mathfrak{g})$ is injective, so the map 
\[\sigma_1\circ{\rm ad}:\mfg\to\Rb\] 
is a norm on $\mathfrak{g}$. Hence there exists $C_1 > 1$ such that 
$$
\norm{X} \le  C_1 \sigma_1( {\rm ad}(X) )
$$
for all $X \in \mathfrak{g}$. Then, since ${\rm ad}(\mathfrak{a})$ is a subgroup of the diagonal matrices, 
$$
\norm{Y} \le C_1 \sigma_1( {\rm ad}(Y) )=C_1\log \sigma_1({\rm Ad}(e^Y))
$$
for all $Y \in \mathfrak{a}$. Hence, by the $\mathsf{KAK}$-decomposition, 
\begin{align*}
\norm{ \kappa(g)} \le C_1  \log \sigma_1({\rm Ad}(e^{\kappa(g)})) \le C_1C_2^2  \log \sigma_1({\rm Ad}(g))
\end{align*}
for all $g \in \GG$, where
$$
C_2 : = \max_{k \in \Ksf}\, \sigma_1( {\rm Ad}(k) ).
$$
So by Lemma~\ref{lem:estimate on sigma1 when linear} and part (3),
\begin{align*}
\norm{\kappa(e^Y)} &\le C_1C_2^2  \log \sigma_1(e^{{\rm ad}(Y)})\le C_1C_2^2\log C_0+\frac{C_1C_2^2}{2}\log P({\rm ad}(Y))
\end{align*}
for all $Y \in \mathfrak{u}$. Thus, to prove (4), it suffices to show that there is some $A,a>0$ such that 
\begin{align}\label{eqn: dfkgokdf}
P({\rm ad}(Y))\le A(1+\norm{Y})^a
\end{align}
for all $Y\in\mfu$.

Again, since $Z(\mathsf{G})$ is trivial, the map $\norm{\cdot}':\mathfrak g\to\Rb$ given by
\[\norm{Y}':=\max_{1\le i,j\le d}\abs{[{\rm ad}(Y)]_{i,j}}.\]
is a norm, and so is bilipschitz to $\norm{\cdot}$. At the same time, observe that there is a polynomial function $Q$ of one variable with positive coefficients such that 
\[P({\rm ad}(Y))\le Q(\norm{Y}')\]
for all $Y\in\mathfrak u$. Observe that there is some $B,b>0$ such that
\[Q(\norm{Y}')\le B(1+\norm{Y}')^b\]
for all $Y\in\mathfrak u$, so Equation \eqref{eqn: dfkgokdf} holds.
 So (4) follows.

Proof of (5). For $\psi \in \sum_{\alpha \in \Delta} \Zb_{\ge 0} \cdot \omega_\alpha$ let $\chi_\psi : = \psi + \sum_{\alpha \in \Delta}  \omega_\alpha$. By Proposition~\ref{prop:reduction to the linear case}, for each such $\psi$ there exist $d_\psi, N_\psi \in \Nb$ and a irreducible representation $\Phi_\psi : \GG \rightarrow \SL(d_\psi,\Rb)$ such that 
$$
e^{\chi_\psi(\kappa(g))} =  \sigma_1(\Phi_\psi(g))^{1/N_\psi }
$$
for all $g \in \GG$. By part (3), $\Phi_\psi(\mathsf{U})\subset\SL(d_\psi,\Rb)$ is unipotent with Lie algebra $\d\Phi_\psi(\mathfrak{u})$. So by Lemma~\ref{lem:estimate on sigma1 when linear} there exist $A_\psi > 1$ and a positive polynomial $P_\psi: \mathfrak{u} \rightarrow \Rb$ such that 
\begin{equation}\label{eqn:bound for chi psi}
\frac{1}{A_\psi} P_\psi(Y)^{1/(2N_\psi)} \le e^{\chi_\psi(\kappa(e^Y))}  \le A_\psi P_\psi(Y)^{1/(2N_\psi)}
\end{equation} 
for all $Y \in \mathfrak{u}$. Since $\omega_\alpha = \chi_{\omega_\alpha} - \chi_0$, (5) follows from Equation~\eqref{eqn:bound for chi psi} with 
$$
R_\alpha(Y): = \frac{P_{\omega_\alpha}(Y)^{N_0}}{P_{0}(Y)^{N_{\omega_\alpha}}},
$$
$M_\alpha := 2N_0 N_{\omega_\alpha}$ and $C_\alpha := A_{\omega_\alpha} A_0$. 
\end{proof} 

\section{Proof of Theorem~\ref{thm:properties of relatively Anosov representations}}\label{appendix:properties of pers}

In this appendix we prove Theorem~\ref{thm:properties of relatively Anosov representations}. As mentioned before, Theorem~\ref{thm:properties of relatively Anosov representations} was established in~\cite{zhu-zimmer1} in the special case when $\GG = \SL(d,\Rb)$. In the following argument we use Proposition~\ref{prop:reduction to the linear case} to reduce to this special case. 

\begin{theorem}Assume $Z(\mathsf{G})$ is trivial and $\Psf_\theta$ contains no simple factors of $\GG$. Suppose $\Gamma\subset \GG$ is a $\Psf_\theta$-Anosov subgroup relative to $\Pc$.
\begin{enumerate}
\item If $X$ is a Groves--Manning cusp space for $(\Gamma, \Pc)$ and $M :=\GG/\Ksf$ is a Riemannian symmetric space associated to $\GG$, then there exist $c > 1$, $C> 0$ such that 
$$
\frac{1}{c} \d_M(\gamma\Ksf, \Ksf)- C \le \d_X(\gamma, \id) \le c \d_M(\gamma\Ksf, \Ksf) + C
$$
for all $\gamma \in \Gamma$. 
\item If $P \in \Pc$, then $P$ is a cocompact lattice in a closed Lie group $\mathsf{H}$ with finitely many components. Moreover,
\begin{enumerate} 
\item $\mathsf{H} = \mathsf{L} \ltimes \mathsf{U}$ where $\mathsf{L}$ is compact and $\mathsf{U}$ is the unipotent radical of $\mathsf{H}$. 
\item $\mathsf{H}^0 = \mathsf{L}^0 \times \mathsf{U}$ and $\mathsf{L}^0$ is Abelian. 
\end{enumerate} 
\end{enumerate} 

\end{theorem}

Fix $\GG$, $\Psf_\theta$, $\Gamma$, and $\Pc$ satisfy the assumptions of Theorem~\ref{thm:properties of relatively Anosov representations}. Let $\chi : = \sum_{\alpha \in \theta} \omega_\alpha$, and let $N \in \Nb$, $\Phi : \GG \rightarrow \SL(d,\Rb)$ and $\xi : \Fc_\theta \rightarrow \Fc_{1,d-1}(\Rb^d)$ satisfy Proposition~\ref{prop:reduction to the linear case} for $\chi$. Let $\GG_\star$ denote the Zariski closure of $\Phi(\GG)$ in $\SL(d,\Rb)$.

\begin{lemma}\label{lem:properties of Gstar} $\Phi$ is injective, $\GG_\star$ is semisimple and $\Phi(\GG)=\GG_\star^0$. 
\end{lemma}

\begin{proof} 
Since $\GG$ is semisimple, $\ker \Phi$ is either discrete or contains a simple factor of $\GG$. Since $\xi : \Fc_\theta \rightarrow \Fc_{1,d-1}(\Rb^d)$ is a $\Phi$-equivariant embedding, $\ker \Phi$ must act trivially on $\Fc_\theta$. So $\ker \Phi \subset \Psf_\theta$. By assumption $\Psf_\theta$ contains no simple factors of $\GG$, so $\ker \Phi$ is discrete. Since $\GG$ is connected, for every $g\in \GG$, there is a continuous path in $\GG$ connecting $\id$ and $g$. Since $\ker \Phi$ is normal, this implies that for any $h\in\ker \Phi$, there is a continuous path in $\ker \Phi$ between $h$ and $ghg^{-1}$. The discreteness of $\ker \Phi$ then implies that $ghg^{-1}=h$. Since both $g$ and $h$ are arbitrary, we see that $\ker \Phi$ is contained in the center of $\GG$, and so $\ker \Phi$ is trivial. 

By construction $\Phi(\GG)\subset\SL(d,\Rb)$ is irreducible and contains a proximal element and hence $\GG_\star$ is a semisimple Lie group by ~\cite[Lem.\ 2.19]{BCLS}. 

Since $\d\Phi(\mathfrak g)$ is the Lie algebra of $\Phi(\GG)$, 
$$
{\rm Ad}(h) \d\Phi(\mathfrak{g}) = \d\Phi(\mathfrak{g}) 
$$
for all $h \in \Phi(\GG)$. So 
$$
{\rm Ad}(h) \d\Phi(\mathfrak{g}) = \d\Phi(\mathfrak{g}) 
$$
for all $h \in \GG_\star$. Since $\GG$ is connected, $\Phi(\GG)$ is a connected normal subgroup of $\GG_\star^0$, and thus is an almost direct product of simple factors of $\GG_\star^0$. 

Suppose for contradiction that there is a simple factor $\mathsf H\subset \GG^0_\star$ that does not lie in $\Phi(\GG)$. Since $\Phi$ is irreducible and $\mathsf H$ commutes with $\Phi(\GG)$, we may apply Schur's lemma to deduce that $\mathsf H\cong\Rb$, which is impossible since $\GG_\star^0$ is semisimple. Thus, $\Phi(\GG)=\GG_\star^0$. 
\end{proof}

Recall that $\d_M$ satisfies $\d_M(g\Ksf, \Ksf) = \norm{\kappa(g)}$ for all $g \in \GG$. So to prove part (1) it suffices to prove the following. 

\begin{lemma} If $X$ is a Groves--Manning cusp space for $(\Gamma, \Pc)$, then there exist $c > 1$, $C> 0$ such that 
$$
\frac{1}{c}\norm{\kappa(\gamma)} - C \le \d_X(\gamma, \id) \le c \norm{\kappa(\gamma)} + C
$$
for all $\gamma \in \Gamma$. 

\end{lemma} 

\begin{proof} Let $N : = \SL(d,\Rb) / \mathsf{SO}(d)$ be the symmetric space associated to $ \SL(d,\Rb)$ and let $x_0 : = \mathsf{SO}(d) \in N$. Since $\Phi(\Gamma)$ is $\Psf_{1,d-1}$-Anosov relative to $\Pc$ and $\ker \Phi$ is trivial, by~\cite[Thm.\ 1.7]{zhu-zimmer1}, there exist $c_1 > 1$, $C_1> 0$ such that 
$$
\frac{1}{c_1}\d_{N}(\Phi(\gamma) x_0, x_0) - C_1 \le \d_X(\gamma, \id) \le c_1 \d_{N}(\Phi(\gamma) x_0, x_0) + C_1
$$
for all $\gamma \in \Gamma$. Then there exist $c_2> 1$ such that 
\begin{equation}\label{eqn:distance in Mprime and sigma1}
\frac{1}{c_2}\log \sigma_1(\Phi(\gamma)) - C_1 \le \d_X(\gamma, \id) \le c_2 \log \sigma_1(\Phi(\gamma))  + C_1
\end{equation} 
for all $\gamma \in \Gamma$. 

Lemma~\ref{lem:properties of Gstar} implies that $\d\Phi(\mfa)$ is a Cartan subspace of the Lie algebra of $\GG_\star$. By~\cite[Thm.\ 7]{Mostow55},  $\d\Phi(\mfa)$ is conjugate to a subspace of the symmetric matrices in $\mathfrak{sl}(d,\Rb)$, which in turn implies that $\d\Phi(\mfa)$ is conjugate to a subspace of the diagonal matrices. So there exists $c_3 > 1$ such that 
$$
\frac{1}{c_3} e^{\sigma_1(\d\Phi(X))} \le \sigma_1(e^{\d\Phi(X)}) \le c_3e^{\sigma_1(\d\Phi(X))}
$$
for all $X \in \mathfrak{a}$. Since $\Phi$ is injective, so is $\d\Phi$. Hence there exists $c_4 > 1$ such that 
$$
\frac{1}{c_4} \sigma_1(\d\Phi(X)) \le \norm{X} \le c_4 \sigma_1(\d\Phi(X))
$$
for all $X \in \mathfrak{g}$. Finally, since $\Ksf$ is compact, 
$$
D : = \max_{k\in\Ksf} \sigma_1(\Phi(k))
$$
 is finite. 
 
 Now if $g \in \GG$, then by the $\mathsf{KAK}$-decomposition,
 \begin{align*}
 \log \sigma_1(\Phi(g)) &\le 2 \log D + \log \sigma_1\left(e^{\d\Phi(\kappa(g))}\right) \\
 &\le 2 \log D +\log c_3+ \sigma_1( \d\Phi(\kappa(g)))\\
 & \le 2 \log D +\log c_3+ c_4 \norm{\kappa(g)}
\end{align*}
and likewise 
 $$
 \log \sigma_1(\Phi(g)) \ge -2 \log D -\log c_3 + \frac{1}{c_4} \norm{\kappa(g)}. 
 $$
  Combining these estimates with Equation~\eqref{eqn:distance in Mprime and sigma1} completes the proof. 
\end{proof}

\begin{lemma}\label{lem:properties of Zariski closure in appendix}
 If $P \in \Pc$, then $P$ is a cocompact lattice in a closed Lie subgroup $\mathsf{H} \subset \GG$ with finitely many components. Moreover,
\begin{enumerate} 
\item $\mathsf{H} = \mathsf{L} \ltimes \mathsf{U}$ where $\mathsf{L}$ is compact and $\mathsf{U}$ is the unipotent radical of $\mathsf{H}$. 
\item $\mathsf{H}^0 = \mathsf{L}^0 \times \mathsf{U}$ and $\mathsf{L}^0$ is Abelian. 
\end{enumerate}
\end{lemma} 

\begin{proof} 

Let $\mathsf{H}_\star$ denote the Zariski closure of $\Phi(P)$ in $\SL(d,\Rb)$. Since $\Phi(\Gamma)$ is $\Psf_{1,d-1}$-Anosov relative to $\Pc$, ~\cite[Prop.\ 4.2 and Thm.\ 8.1]{zhu-zimmer1} imply that  $\Phi(P)$ is a cocompact lattice in $\mathsf{H}_\star$. Moreover 
\begin{enumerate} 
\item $\mathsf{H}_\star = \mathsf{L}_\star \ltimes \mathsf{U}_\star$ where $\mathsf{L}_\star$ is compact and $\mathsf{U}_\star$ is the unipotent radical of $\mathsf{H}_\star$. 
\item $\mathsf{H}^0_\star = \mathsf{L}^0_\star \times \mathsf{U}_\star$ and $\mathsf{L}^0_\star$ is Abelian. 
\end{enumerate}
Then $\mathsf{U} := \Phi^{-1}(\mathsf{U}_\star)$, $\mathsf{L}: =\Phi^{-1}(\mathsf{L}_\star \cap \mathsf{G}_\star^0)$ and $\mathsf{H} := \Phi^{-1}(\mathsf{H}_\star \cap \mathsf{G}_\star^0)$ satisfy the lemma.  
\end{proof}

\end{document}